\documentclass[a4paper,12pt]{article}
\usepackage[utf8]{inputenc}
\usepackage{amsmath}
\usepackage{amssymb}
\usepackage{amsthm}
\usepackage{xcolor}

\usepackage{enumerate}
\usepackage{amssymb, amsmath,graphicx}
\usepackage{todonotes}
\usepackage{verbatim}
\usepackage{float}

\theoremstyle{plain}
\newtheorem{thm}{Theorem}[section]
\newtheorem{theorem}[thm]{Theorem}
\newtheorem{proposition}[thm]{Proposition}
\newtheorem{definition}[thm]{Definition}
\newtheorem{lemma}[thm]{Lemma}
\newtheorem{corollary}[thm]{Corollary}

\newtheorem{remark}[thm]{Remark}
\newtheorem{example}[thm]{Example}

\newtheorem*{example*}{Example}
\newtheorem*{remark*}{Remark}

\newcommand{\supp}{{\sf supp}}

\newcommand{\Ric}{\mathrm{Ric}}

\def\D{{\mathcal D}}

\newcommand{\E}{\mathcal{E}}

\newcommand{\diam}{\textrm{diam}}

\newcommand{\loc}{{\rm loc}}

\newcommand{\BE}{{\sf BE}}
\newcommand{\pBE}{{\sf pBE}}

\newcommand{\CD}{{\sf CD}}

\def\N{{\mathbb N}}
\def\Z{{\mathbb Z}}
\def\R{{\mathbb R}}
\def\C{{\mathbb C}}

\usepackage{amsmath, verbatim, amsfonts, amssymb, color}
\usepackage{mathrsfs}
\usepackage{dsfont}
 \usepackage{mathbbol}

\def\M{{\sf M}}
\DeclareMathOperator{\vol}{\mathsf{vol}}

\DeclareMathOperator{\Sec}{\mathrm{Sec}}

\def\C{\mathcal{C}}

\def\N{{\mathbb N}}    
\def\Z{{\mathbb Z}} 
\def\R{{\mathbb R}}

\begin{document}

\title{\bfseries  Bakry-\'Emery, Hardy, and Spectral Gap Estimates on Manifolds with Conical Singularities}

\author{Karl-Theodor Sturm
\\[1cm]
\small Hausdorff Center for Mathematics \& Institute for Applied Mathematics\\
\small University of Bonn, Germany 
}

\maketitle



\abstract
We study spectral properties and geometric functional inequalities on Riemannian manifolds of dimension $\ge3$ with singularities.
Of particular interest will be manifolds with (finite or countably many) conical singularities $\{z_i\}_{i\in\mathfrak I}$ in the neighborhood of which the largest lower bound for the Ricci curvature is 
\begin{equation}\label{d2}
k(x)\simeq K_i-\frac{s_i}{d^2(z_i,x)}.
\end{equation}
Thus none of the existing Bakry-\'Emery inequalities or curvature-dimension conditions apply. In particular, $k$ does not belong to the Kato (or extended Kato) class, and $(M,g)$ is not tamed in the sense of \cite{ERST}.
Manifolds with such a singular Ricci bound \eqref{d2} appear quite naturally.
The prime examples are 
\begin{itemize}
\item metric cones, for instance,
 $M=\R_+\times_r N$ with any $(N,g^N)$ satisfying $\inf_{y\in N}\Ric_y^N<(n-2)g^N$,  e.g.~spheres $N={\mathbb S}^{n-1}_R$ with radius $R>1$.
 \end{itemize}
For manifolds with such conical singularities we will prove
\begin{itemize}
\item[-] a version of the Bakry-\'Emery inequality 
\item[-] a novel Hardy inequality
\item[-] a spectral gap estimate.
\end{itemize}
Related examples 
are provided by \begin{itemize}
\item weighted spaces, e.g.
 $M=\R^n$ with $g=g^{Euclid}$ and $m(dx)=|x|^\alpha d{\frak L}^n(x)$ for some $\alpha\in\R$ where the largest lower bound for Bakry-\'Emery Ricci tensor is given by
$
k(x)=-\frac{|\alpha|}{|x|^2}$, and 
%
\item 
Grushin-type spaces $M=\R^j \times_f \R^{n-j}$ with $f(y)=|y|^{-\alpha}$ for suitable $\alpha>0$,
  either with Riemannian volume measure or with Lebesgue measure, 
which admit  lower Ricci bounds of the form
$k(y,z)=-\frac{C}{|y|^2}$.
\end{itemize}
\section{Introduction}

\paragraph{i) Lichnerowicz Inequality, Aubry's Estimate and Related Results.}
The well-known Lichnerowicz inequality \cite{Lichner} provides a sharp lower bound for the spectral gap of $-\Delta$ in terms of a uniform lower bound of the Ricci curvature:
$$\lambda_1\ge \frac{n}{n-1}\,K.$$
with $K:=\inf_{x\in M}\Ric^M_x$.
It  has been extended to Markov semigroups and metric measure spaces with synthetic Ricci bounds $\BE(K,n)$ and $\CD(K,n)$, resp.

Important extensions of Lichnerowicz' inequality have been obtained where the uniform lower bound 
$k\ge K$ is replaced by suitable norms on $(k-K)^-$.

\begin{proposition}[Aubry \cite{Aubry}, Carron--Rose \cite{Carron-Rose}] Assume that 
$k\ge K-v$
with $v\ge0$ in $L^p$ for some $p>n/2$ or, more generally, with $v$ in the Kato class. Then
\begin{equation*}
\lambda_1\ge \frac n{n-1}K-c\,\|v\|_{*}.
\end{equation*}
where $\|\,.\,\|_{*}$ denotes the $L^p$-norm or a suitable Kato norm.
\end{proposition}

These criteria, unfortunately, do not cover the important class of conical singularities in dimension $n\ge3$. 
The Ricci curvature in the tip of a cone over a large sphere decays like $\frac{-c}{d^2(z,.)}$. These functions are never in $L^p$ for $p>n/2$, and never in the Kato class.

\paragraph{ii) The Spectral Gap Estimate.}
Our main result is a spectral gap estimate for manifolds $(M,g)$ with (finite or countably many)  conical singularities.
Let us present it in its most simple form.

\begin{theorem} 
Assume that    
\begin{itemize}
\item $(M,g)$ is smooth on $M\setminus \{z_i\}_{i\in\mathfrak I}$ and
$\Ric^M\ge K$ on $M\setminus \bigcup_{i\in\mathfrak I} B_{\rho}(z_i)$,
\item $B_{\rho}(z_i)\simeq [0,\rho)\times_{\sin_{\ell}} N_i$  for some manifolds $N_i$ with $\Ric^{N_i}\ge (n-2)\kappa$
\end{itemize}
for some $K,\rho,\ell\in\R$ with $\sqrt\ell\rho=\frac\pi2$ and  
$1\ge\kappa\ge \frac{6-n}4$.
 Then
\begin{align}\lambda_1&\ge
\min\Big\{K, \ \big(n\kappa-1\big)\ell\Big\}.
\end{align}
\end{theorem}
A slightly more sophisticated (and improved) estimate even holds under the weaker assumption
  \begin{equation}
  \label{cccrit}
  \kappa>\frac{n(6-n)-4}{4(n-1)}
  \end{equation}
  which proves crucial for various estimates. It is satisfied 
  for any $\kappa\ge0$ if $n\ge6$, and
  for
  $$\kappa>\begin{cases}
  1/{16}, \quad&\text{ if }n=5,\\
  1/3, \quad&\text{ if }n=4,\\
  5/8, \quad&\text{ if }n=3.
  \end{cases}$$
  In particular, for $n\ge6$ this includes cones over arbitrarily large spheres.
  In any of these cases, from a spectral theoretic perspective, the contribution of a conical singularity can be regarded as a `small perturbation'.
  
  \paragraph{iii) The Road Map.}
  Our strategy for deriving such a spectral gap estimate is based on the following fundamental results:
\begin{itemize}
\item for manifolds with conical singularities the pointwise Ricci bound
$$\Ric_x^M\ge k(x):= K-\sum_i\frac{s_i}{d^2(z_i,x)} \quad\text{on }M\setminus\{z_i\}_i$$
implies the \emph{mild Bakry-\'Emery estimate} $\BE_1'(k,n)$

\item $\BE_1'(k,n)$ implies the \emph{spectral estimate} 
$$\lambda_1(-\Delta)\ge \lambda_0(-\Delta +k)$$
and a $n$-dependent improved version
\item the \emph{Hardy inequality} on manifolds with conical singularities asserts that
$$-\Delta \ge \vartheta\quad\text{for some function}\quad \vartheta\ge \sum_i\frac{t_i}{d^2(z_i,.)}-C.$$
\end{itemize}
Thus $\lambda_1\ge K-C$ provided $s_i\le t_i$ for all $i$.

Let us discuss these ingredients in more detail.

 \paragraph{iv) The Mild Bakry-\'Emery Inequality.}


Bakry-\'Emery Inequalities
provide lower bounds on the $\Gamma_2$-operator
$$\Gamma_2(u):=\frac12\Delta|\nabla u|^2-\nabla u\, \nabla\Delta u$$
in terms of the lower bound $k$ for the Ricci curvature of the underlying manifold $(M,g)$ and the upper bound $N$ for its dimension.
The most prominent ones are the pointwise (`classical')  $\pBE_2(k,N)$ inequality
$$\Gamma_2(u)\ge k |\nabla u|^2 +\frac1N (\Delta u)^2$$
and the (`self-impoved')  $\pBE_1(k,\infty)$ inequality
$$\Gamma_2(u)\ge k |\nabla u|^2 +\big| \nabla |\nabla u|\big|^2.$$
The $\pBE^M_1(k,N)$ inequality improves up on both of them:
\begin{align*}
\Gamma_2(u)&\ge k|\nabla u|^2+\frac1N (\Delta u)^2+\frac{N}{N-1}\,\Big| \nabla\big|\nabla u|\big|-\frac1N|\Delta u| \Big|^2\\
&=k|\nabla u|^2+\big|\nabla|\nabla u|\big|^2+\frac{1}{N-1}\,
\Big| \big|\nabla|\nabla u|\big|-|\Delta u| \Big|^2.
\end{align*}
Typically, these estimates will not be requested pointwise but only  in the weak form which in the latter case leads to the mild Bakry-\'Emery Inequality $\BE_1'(k,N)$:
$$\int\varphi \Gamma_2(u)dm\ge \int\varphi\bigg[
k|\nabla u|^2+\big|\nabla|\nabla u|\big|^2+\frac{1}{N-1}\,
\Big| \big|\nabla|\nabla u|\big|-|\Delta u| \Big|^2
\bigg]dm$$
for all $u$ and $\varphi\ge0$ in suitable  function spaces which merely require weak derivatives of $u$ and $\varphi$ of order 2 and 1 (whereas the `standard' versions of $\BE(k,N)$ require weak derivatives of $u$ and $\varphi$ of order 3 and 2). 
%
We discuss its relation to other versions of $\BE(k,N)$ and illustrate how to verify it for manifolds with singularities.

\newpage
We address the following fundamental questions:
\begin{itemize}
\item[(i)] Under which conditions can one conclude
$$\Ric^M\ge k\quad\text{on }M\setminus \{z\}\quad\stackrel?\Longrightarrow\quad
\BE'_1(k,n)$$
\item[(ii)] Does $\BE_1'(k,n)$ hold for manifolds with conical singularities?
\end{itemize}
Caution is required since there is a `counterexample'.

%
%
\begin{example} Let $M:=\R_+\times_r N$ be the cone over $N=\R/(4\pi\Z)$, the circle of length $4\pi$. Or in other words, $M$ is the twofold cover of $\R^2\setminus\{0\}$ together with the origin $\{0\}$.
 Equivalently, $M$ can be identified  with $\R^2$ equipped with the metric
 $g^M_x=|x|^4\, g_x^{Euclid}$.
 
 Then $\Ric^M_x\equiv0$ on $M\setminus\{0\}$ but $\BE_1'(0,n)$ does not hold nor does $\BE_2(K,N)$ for any $K$ and $N\in\R$.
\end{example}

We will give affirmative answers to both of the above questions (i) and (ii) in case ($M,g)$ is a manifold with (finite or countably many) conical singularities $\{z_i\}_i$ such   that
for each $i$,
$$M\supset B_{\rho_i}(z_i)\simeq [0,\rho_i)\times_{f_i} N_i$$ with  $f_i(r)=\sin_{\ell_i}(r)$ for some   $\ell_i\in\R$, some $\rho_i>0$, and complete Riemannian manifolds $N_i$.

\begin{theorem} \label{thm-be}
Assume that $n\ge 3$, that $g$ is smooth on $M\setminus \{z_i\}_i$ with
$$\Ric^M\ge k\quad\text{on }M\setminus \{z_i\}_i,$$
and that
   $\Ric^{N_i}\ge (n-2){\kappa} \  (\forall i)$ with
   \begin{equation}
  \label{ccrit}
\kappa>\frac{n(6-n)-4}{4(n-1)}.
  \end{equation}
Then 
$\BE'_1(k,n)$ is satisfied.
  \end{theorem}

   \paragraph{v) A Powerful Spectral Estimate.} 
%
%
 Besides our novel Hardy inequality, the key ingredient is an estimate of the spectral gap $\lambda_1:=\inf\big(\text{\rm spec}\left(-\Delta\right)\setminus\{0\}\big)$ for the Laplacian $\Delta$ in terms of the spectral bound for the Schr\"odinger operator $-\Delta+k$ where the lower Ricci bound $k$ plays the role of a potential.

\begin{theorem}
Assume $\BE_1'(k,n)$ and that the spectrum of $\Delta$ is discrete.  Then 
\begin{equation}
\lambda_1\ge
\inf\text{\rm spec}\left(- \Delta+k\right).
\end{equation}
Moreover generally,
for every $t\in[0,1)$,
\begin{equation}\label{sixteen}
\lambda_1\ge\alpha_{t}\cdot
\inf\text{\rm spec}\left(-t\frac{n}{n-1} \Delta+k\right)
\quad\text{with} \quad
\alpha_{t}:=\frac n{n-1}\bigg[1+\frac{\frac t{1-t}}{(n-1)^2}
\bigg]^{-1}.
\end{equation}
\end{theorem}

%
The borderline case $t=0$ 
provides  the classical Lichnerowicz inequality:
$$\lambda_1\ge \frac n{n-1} \inf_xk(x).$$  

\begin{remark}\begin{enumerate}[\rm(i)]
\item The estimate \eqref{sixteen} holds true in the more general setting of infinitesimally Hilbertian metric measure spaces with curvature-dimension condition \CD$(k,n)$  or Markov semigroups with  Bakry-\'Emery condition \BE$(k,n)$ for lower bounded, variable $k:M\to\R$  and  $n\in [1,\infty]$.

\item The estimate \eqref{sixteen} improves upon a similar  estimate in  \cite[ Lemma 2.5]{Carron-Rose}.

\item Actually, Carron--Rose \cite[Prop. 2.6]{Carron-Rose} claim to have a much better estimate:
$\lambda_1\ge\frac{n}{n-1}\cdot
\inf\text{\rm spec}\left(-\frac{n}{n-1} \Delta+k\right).$
The proof of that result, however, is not correct. Their lower bound on the Ricci curvature of the cone over $M$ is not true in radial direction.
\end{enumerate}
\end{remark}
%
%
 \paragraph{vi) The Hardy Inequality.} The celebrated Hardy inequality \cite{Hardy} on $\R^n$ states that for every $z\in\R^n$,
 \begin{equation*}\label{null}
-\Delta\ge \left(\frac{n-2}2\right)^2\, \frac1{|\,.-z|^2}\end{equation*}
in the sense of self-adjoint operators, or more explicitly,
\begin{equation*}\label{zwei}
\int_{\R^n} |\nabla f|^2dx\ge \left(\frac{n-2}2\right)^2\, \int_{\R^n} \frac{f^2(x)}{|x-z|^2}\,dx
\quad\qquad\forall f\in W^{1,2}(\R^n).
\end{equation*}

\medskip

 What is the Riemannian counterpart?
 
  \begin{theorem} 
  Let $(M,g)$ be a  complete smooth Riemannian manifold  with dimension $n\ge3$, sectional curvature $\le \ell$, and injectivity radius $> R_\ell:=\frac\pi{2\sqrt\ell}$ for some $\ell>0$.
Then for every $z\in M$, \begin{equation}\label{eins}
-\Delta\ge \left(\frac{n-2}2\right)^2\, \frac1{\underline\tan^2_\ell d(z,.)}-\frac{n-2}2\ell\,{\bf 1}_{B_{R_\ell}(z)}
\end{equation}
where  $\underline\tan_\ell(r):=\frac1{\sqrt\ell}\,\tan\big(\sqrt\ell r \wedge\pi/2\big)$.
An analogous result holds in the case $\ell\le0$.

Moreover, the same estimate holds true for a not necessarily smooth Riemannian manifold $(M,g)$ if it is a ${\sf CAT}(\ell)$-space.
%

\end{theorem}
For $\ell=0$, 
we recover the result of \cite{Carron} on spaces with nonpositive curvature.
\begin{corollary}  If $M$ is a smooth Cartan-Hadamard manifold or a ${\sf CAT}(0)$-space  then for every $z\in M$,
$$-\Delta\ge \left(\frac{n-2}2\right)^2\, \frac1{d^2(x,z)}.$$
\end{corollary}

If $M$ has a conical singularity, then proving a Hardy estimate at the singularity  does not require an upper bound on the sectional curvature but merely on the convexity of 
 the warping function. 

\begin{theorem}
Assume that $M$ has a conical singularity at $z$ such that
$$M\supset 
B_{\rho}(z)\simeq [0,\rho)\times_{f} N$$ 
with 
 $f(r)=\sin_{\ell}(r)$ for some   $\ell\in\R$.
Then for every $L\ge\ell\vee (\frac\pi{2\rho})^2$,
$$-\Delta\ge 
 \left(\frac{n-2}2\right)^2\, \frac{1}{\underline\tan_{L}^2(d(z_i,.))}-\frac{n-2}2L {\bf 1}_{B_{R_L}(z)}.$$
\end{theorem}
Note that  no assumption on $N$ and no assumption on $M\setminus B_\rho(z)$ is made.
\paragraph{vii) Setting and Notation.}

Most of the concepts and results easily carry over to metric measure spaces with synthetic Ricci curvature bounds \CD$(k,n)$ and to Markov semigroups satisfying a Bakry-\'Emery condition \BE$(k,n)$ with variable $k:M\to\R$  and a real number $n\in [1,\infty]$. To keep the presentation as simple as possible, however, we confine ourselves to present them in the setting of Riemannian manifolds $(M,g^M)$ with not necessarily smooth metric tensors.
 
 The Ricci tensor $\Ric^M$ will only be considered on open subsets $M_0\subset M$ on which $g^M$ is (sufficiently) smooth.
As long as misunderstanding is excluded, in the sequel we will mostly write $\Ric^M\ge k$ instead of $\Ric^M\ge k g^M$. 

\section{Manifolds with Singularities}
\subsection{Warped Products}



The basic examples of singular spaces to be considered in this paper  will be Riemannian manifolds which locally around these singularities look like warped products.

Given   a smooth $(n-1)$-dimensional Riemannian manifold $(N,g^N)$, an interval $I=[0,\rho)$ for some $\rho\in(0,\infty]$, and a
$\C^\infty$-function $f:I\to\R_+$  with $f(0)=0$ and $f(r)>0$ for $r>0$,
we consider the \emph{warped product} 
$$M=I\times_f N,\qquad dg^M=dr^2+f^2(r)\,dg^N.$$
More formally, $M:= (I\times N)/\sim$ with $(r,y)\sim(r',y')$ if $r=r'=0$,
and for $x=(r,y)\in M
$  and  $\zeta=\tau+\xi\in T_x\simeq \R\times T_yN$,
$$g^M(\tau+\xi)=|\tau|^2+ f^2(r)\, g^N(\xi).$$
\begin{lemma}[{\cite{Oneill}, \cite{Ketterer-cones}}] 
\label{warp-form}
$(M\setminus\{0\}, g^M)$ is a smooth Riemannian manifold, and for $r\not=0$,
\begin{align*}\Ric^M_{r,y}(\tau+\xi)
&=-(n-1)\frac{f''(r)}{f(r)}|\tau|^2+\Ric^N_y(\xi)-\left(\frac{f''(r)}{f(r)}+(n-2)\left|\frac{f'(r)}{f(r)}\right|^2\right)\, g^M_{r,y}(\xi)\\
&=-(n-1)\frac{f''(r)}{f(r)}|\tau|^2+\Ric^N_y(\xi)-\left(\frac{f''(r)}{f(r)}+(n-2)\left|\frac{f'(r)}{f(r)}\right|^2\right)\, f^2(r) \,g^N_y(\xi).
\end{align*}
\end{lemma}
%

%
\begin{corollary} Assume that 
$\Ric^N\ge (n-2)\kappa \ g^N$ for some $\kappa\in\R$.
Then for all $(r,y)\in M$ with $r\not=0$ and all $\zeta\in T_{r,y}M$,
$$\Ric^M_{r,y}(\zeta)\ge k(r)\,g^M_{r,y}(\zeta)$$  with
\begin{align}\label{ric-warp}\nonumber
k(r)&=\min\left\{-(n-1)\frac{f''}f, \ (n-2)\frac{\kappa-{f'}^2}{f^2}-\frac{f''}f\right\}(r)\\
&=-(n-1)\frac{f''}f (r)-(n-2)\left(\frac{\kappa-{f'}^2}{f^2}+\frac{f''}f\right)^- (r).
\end{align}
\end{corollary}

\subsection{Cones over Large Spheres}

The prime example of a warping function is
 \begin{equation}\sin_\ell(r)=\begin{cases}
\frac1{\sqrt\ell}\sin(\sqrt\ell r), \qquad &\text{if }\ell>0
\\
 r, \qquad &\text{if }\ell=0\label{sin-ell}\\
\frac1{\sqrt{-\ell}}\sinh(\sqrt{-\ell} r),\qquad &\text{if }\ell<0
\end{cases}
\end{equation}
with $I\subset 
[0,2R_\ell)$.

If $f=\sin_\ell$ for some $\ell\in\R$ then the warped product $M=I\times_f N$ is called \emph{$\ell$-cone} with basis $N$.
In this case, \eqref{ric-warp} amounts to
\begin{equation}\label{ric-cone}
k(r)=(n-1)\ell -(n-2)\frac{(1-\kappa)^+}{\sin^2_\ell(r)}.\end{equation}

%
%
%
The standard example here is
\begin{example} Assume that $N={\mathbb S}^{n-1}$ is the round sphere of radius 1 (thus $\kappa=1$), $f=\sin_\ell$ and $I=[0,2R_\ell)$. Then $\Ric^M=(n-1)\ell\, g^M$. 

In particular
\begin{itemize}
\item 
If $\ell<0$ then $M$ is the hyperbolic space with curvature $\ell<0$.
\item 
If $\ell>0$ then  $M$ is the round  $n$-sphere of radius $R=\frac1{\sqrt\ell}$ and curvature $\ell>0$. (More precisely, $M$ is the punctured round sphere:  the antipodal point of the vertex is excluded here for convenience.)
\item If $\ell=0$ then  $M$ is the Euclidean space $\R^n$.
%
\end{itemize}
\end{example}

More exotic examples are
\begin{example}Consider the cone $M=\R_+\times_r N$ with
$N={\mathbb S}^{2}_{1/\sqrt 3}\times {\mathbb S}^{2}_{1/\sqrt 3}$.
Then 
\begin{itemize}
\item $M$ has nonnegative Ricci curvature in synthetic sense, i.e.~it satisfies the curvature-dimension condition \CD$(0,5)$ and the Bakry-\'Emery condition  $\BE_2(0,5)$;
\item $M\setminus\{0\}$ has nonnegative Ricci curvature in classical sense, i.e. $\Ric^M\ge0$ on $M\setminus\{0\}$;
\item $M\setminus\{0\}$ has unbounded sectional curvature: for all $(r,y)\in M\setminus\{0\}$ and $\zeta\in T_{r,y}M$,
$$\sup_{
\rho\perp\zeta}\  \Sec^M_{r,y}(\zeta, \rho)=\frac2{r^2}, \qquad 
\inf_{
\rho\perp\zeta}\ \Sec^M_{r,y}(\zeta,\rho)=-\frac1{r^2}.
$$
\end{itemize} \end{example}
\begin{example} Consider the cone  $M:=\R_+\times_r N$ with $N={\mathbb{RP}}^2={\mathbb S}^2/\sim$ where $x\sim y$ if $x=-y$. 
Then
\begin{itemize}
\item $N$ is a closed, smooth Riemannian manifold with $\Sec^N=1$ and $\Ric^N=g^N$;
\item $M\setminus\{0\}=(\R^3\setminus\{0\})/\sim$ with $\sim$ as before; 
\item $\Sec^M=0$ and $\Ric^M=0$ on $M\setminus\{0\}$;
\item $M$ has has nonnegative Ricci curvature in synthetic sense, i.e.~it satisfies the curvature-dimension condition $\CD(0,3)$ and the Bakry-\'Emery condition  $\BE_2(0,3)$;
\item $M$ has nonnegative sectional curvature in the sense of Alexandrov.
\end{itemize}
\end{example}
\begin{example} Assume that  $\Ric^N\ge (n-2)\kappa \ g^N$ with $\kappa<1$ 
 and that $f(r)=r$.
Then the warped product $M=[0,\rho)\times_f N$ is a cone 
with lower Ricci bound
\begin{equation}\label{ric-sph}k(r)=-(n-2)\frac{1-\kappa}{r^2}.\end{equation}
\end{example}
If $N=\mathbb S^{n-1}_R$ is a sphere of radius $R=\frac1{\sqrt \kappa}>1$, the cone $M=[0,\rho)\times_r N$ is called `cone over a large sphere'.

\subsection{Manifolds with Conical Singularities}

We aim in the sequel for spectral estimates on singular spaces with particular focus on manifolds  
with a finite or infinite number of conical singularities.
\begin{definition}\label{man-con}
 A Riemannian manifold $(M,g^M)$ (with a not necessarily smooth metric tensor)  is called \emph{Riemannian manifold with conical singularities} if there exist
\begin{itemize}
\item[-]  a discrete 
set of singularities 
$\{z_i\}_{i\in\mathfrak I}\subset M$, 
\item[-] numbers $K,\kappa_i,\ell_i\in\R$ and $\rho_i\in(0,R_\ell]$,
\item[-]   smooth $(n-1)$-dimensional 
Riemannian manifolds $(N_i,g^{N_i})$ with   $\Ric^{N_i}\ge (n-2)\kappa_i g^{N_i}$ and $\vol(N_i)<\infty$, 
\end{itemize}
such that
 \begin{itemize}
 \item $g$ is smooth on $M\setminus \{z_i\}_{i}$ and 
 $\Ric^M\ge K g^M$ on $M_0:=M\setminus \bigcup_i B_{\rho_i}(z_i)$,
\item $M_i:=B_{\rho_i}(z_i)\simeq [0,\rho_i)\times_{\sin_{\ell_i}} N_i$ .
\end{itemize}
\end{definition}
For convenience, we assume that the sets $M_i$  for $i\in \mathfrak I$ are pairwise disjoint.
Note that $M_i$  has a synthetic upper sectional bound 
if and only if $\sec^{N_i}\le 1$, it has a synthetic lower sectional bound 
if and only if $\sec^{N_i}\ge 1$, and it has a synthetic lower Ricci bound 
if and only if $\Ric^{N_i}\ge (n-2)g^{N_i}$.
If $\ell_i=0$ then each of these synthetic curvature bounds for $M$ will be 0 provided it is finite.

%

\section{Bakry-\'Emery 
Inequalities}

\subsection{Pointwise Bakry-\'Emery Inequalities}
Assume that $(M,g)$ is a Riemannian manifold with a metric tensor which is smooth on an open subset $M_0\subset M$, and that
\begin{equation}
\Ric^M\ge k\quad \text{on }M_0
\end{equation}
for some measurable function $k:M\to\R$.
Bochner's identity asserts that 
\begin{equation}
\Gamma_2(u):=\frac12\Delta|\nabla u|^2-\nabla u\, \nabla\Delta u=\Ric^M(\nabla u,\nabla u)+\|\nabla^2 u\|_{HS}^2\quad \text{on }M_0
\end{equation}
for $u\in\C^\infty(M_0)$. 
As a straightforward consequence therefore the \emph{pointwise 
 Bakry-\'Emery inequality   on $M_0$}, briefly $\pBE_2^{M_0}(k,N)$, holds true for every $N\in [n,\infty]$:
\begin{equation}\label{BE2}
\Gamma_2(u)\ge k|\nabla u|^2+\frac1N (\Delta u)^2\quad \text{on }M_0\qquad\quad(\forall u\in\C^\infty(M_0)).
\end{equation}

The pointwise Bakry-\'Emery inequality $\pBE_2^{M_0}(k,N)$ has an impressive self-improvement property 
(cf. \cite{Ba85},   \cite{BQ00}, \cite{Savare}, \cite{Sturm}).

\begin{lemma} $\pBE_2^{M_0}(k,N)$  implies that for all $u,v,w\in\C^\infty(M_0)$,
\begin{equation}\label{barGamma2improved}
\Gamma_2(u)-k|\nabla u|^2-\frac1N (\Delta u)^2\ge2\,\frac{\left[\nabla v\,\nabla^2u\,\nabla w-\frac1N\Delta u\cdot \nabla v\nabla w\right]^2}{|\nabla v|^2|\nabla w|^2+ \frac{N-2}{N}|\nabla v\nabla w|^2}.
\end{equation}
\end{lemma}
\begin{proof}[{Proof (see \cite{Sturm}, Thm. 3.1).}]
For any given $u\in\C^\infty(M_0)$ and  $x_0\in M_0$, apply \eqref{BE2} at $x_0$ to the function
$ \tilde u:=u+t\left[ v-v(x_0)\right]\left[w-w(x_0)\right]$. 
Observe that $\nabla u=\nabla \tilde u$ at given $x_0$ and 
$$\Delta \tilde u=\Delta u+2t \nabla v\nabla w, \quad \Gamma_2(\tilde u)=\Gamma_2(u)+4t \nabla v\nabla^2u \nabla w+2t^2\left[(\nabla v \nabla w)^2+|\nabla v|^2\, |\nabla w|^2\right].$$
Choosing the  optimal $t$ then yields the claim.
\end{proof}
%
%
%
%
%

\begin{corollary}\label{p-BE1} For all $u$,
\begin{equation*}
\Gamma_2(u)-k|\nabla u|^2-\frac1N (\Delta u)^2\ge\frac{N}{N-1}\,\left| \nabla|\nabla u|-\frac1N\Delta u \,\frac{\nabla u}{|\nabla u|}\right|^2
\end{equation*}
or, equivalently,
\begin{equation}\label{elf}
\Gamma_2(u)-k|\nabla u|^2-\big|\nabla|\nabla u|\big|^2\ge\frac{1}{N-1}\,
\left| \nabla|\nabla u|-\Delta u \,\frac{\nabla u}{|\nabla u|}\right|^2.
\end{equation}
\end{corollary}

\begin{proof} At any given point $x_0\in M_0$, apply the previous Lemma with $v(.):=u(.)$ and  $w(.):=\frac12|\nabla u|^2(.)-\frac1n\Delta u(x_0)\, u(.)$. \end{proof}

We say that the
 \emph{pointwise self-improved
 Bakry-\'Emery inequality   on $M_0$} holds true, briefly $\pBE_1^{M_0}(k,N)$, if
\begin{equation}
\Gamma_2(u)\ge k|\nabla u|^2+
\big|\nabla|\nabla u|\big|^2+\frac{1}{N-1}\,
\Big| \big|\nabla|\nabla u|\big|-\big|\Delta u\big| \Big|^2
\quad \text{on }M_0\qquad\quad(\forall u\in\C^\infty(M_0)).
\end{equation}

\subsection{The Mild Bakry-\'Emery Inequality}

Let $(M,g)$  be a complete Riemannian manifold with not necessarily smooth Riemannian tensor $g$ and dimension $n\ge 2$.

\begin{definition} Given an extended number $N\in [1,\infty]$, we say that a function $k:M\to\R$ is \emph{$N$-admissible} if  it is measurable and if $k^-$ is \emph{form bounded} w.r.t. $-\Delta$ with form bound $<\frac N{N-1}$ in the sense that
 $\exists C,C'\in\R$ s.t. $C<\frac N{N-1}$ and
$$\int k^- v^2dm\le C\int |\nabla v|^2dm+C'\int v^2dm \qquad (\forall v\in \D(\E)).$$
\end{definition}
Without restriction, in the sequel  we also may assume that $k^+$ is bounded.

\begin{definition} Given any measurable $k:M\to\R$, we say that the \emph{mild Bakry-\'Emery inequality} $\BE'_1(k,N)$ holds true on $(M,g)$  if \ $\forall u\in\D(\Delta): \ v:=|\nabla u|\in\D(\E)$ and 
 $\forall  \varphi\in\D_\loc(\E)\cap L^\infty_+$ with $|\nabla \varphi|\in L^\infty$:
  \begin{align}\label{mBE}\nonumber
 \int\nabla\varphi \Big[-\frac12\nabla (v^2) &+\nabla u\, \Delta u\Big] dm+\int\varphi(\Delta u)^2dm+\int k^-\varphi v^2dm\\
 &\ge \int k^+\varphi v^2dm+
 \int\varphi|\nabla v|^2dm+\frac1{N-1}\int\varphi
\Big| |\nabla v|- |\Delta u|\Big|^2dm.
 \end{align}
\end{definition}

\begin{lemma}\label{Hess-Lapl}  $\BE'_1(k,N)$ with $N$-admissible $k$ implies that $\exists C,C'\in\R: \forall u\in\D(\Delta): \ |\nabla u|\in\D(\E)$ and
\begin{equation}\label{Hess--Laplace}
\int \big| \nabla|\nabla u|\big|^2dm\le C \int(\Delta u)^2dm +C'\int u^2dm.
\end{equation}
\end{lemma}

\begin{proof}
Chossing $\varphi=1$, ignoring the contribution 
from $k^+$, applying Cauchy-Schwarz to the mixed term of the square,  and employing form boundedness of $k^-$, we obtain for sufficiently small $\delta>0$,
\begin{align*}
\left(1-\frac{1-1/\delta}{N-1}\right) \int \big|\Delta u\big|^2dm&\ge \Big[1+\frac{1-\delta}{N-1}\Big]\int  \big|\nabla |\nabla u|\big|^2dm- \int k^- |\nabla u|^2dm \\
&\ge \epsilon \int  \big|\nabla |\nabla u|\big|^2dm- C\int\big|\nabla u\big|^2dm.
\end{align*}
\end{proof}

\begin{remark} (i) Slightly \emph{stronger} or \emph{weaker} versions of the previous Definition can be obtained by replacing the 
 term  $\frac1{N-1}\int\varphi
\Big| |\nabla v|- |\Delta u|\Big|^2dm$ 
with 
\begin{itemize}
\item the larger term
$$ \frac1{N-1}\int\varphi
\Big| \nabla v- \Delta u \frac{\nabla u}{|\nabla u|}\Big|^2dm$$

\item or the smaller term 
$$\frac{1-\delta}{N-1}\int\varphi|\nabla v|^2dm+\frac{1-1/\delta}{N-1}\int\varphi(\Delta u)^2dm\qquad (\forall \delta>0).$$
\end{itemize}
Both modifications will lead to the same results in the subsequent sections.

(ii) In the previous Definition, the condition  $\forall  \varphi\in\D_\loc(\E)\cap L^\infty_+$ with $|\nabla \varphi|\in L^\infty$ can \emph{equivalently} be replaced by the more restrictive condition  $\forall  \varphi\in\D(\E)\cap L^\infty_+$ with $|\nabla \varphi|\in L^\infty$. Indeed, each $\varphi$ in the former set can be approximated by $\varphi_j$ in the latter set such that
$$\int\varphi_j wdm\to \int\varphi wdm, \qquad \int\nabla\varphi_j wdm\to \int\nabla\varphi wdm\qquad (\forall w\in L^1).$$
For instance, $\varphi_j:=\psi_j \varphi$ with $\psi_j:=1\wedge [j-d(o,.)]\vee0$ for some $o\in M$ will do the job.

(iii) If $k$ is $N$-admissible then in the previous Definition, the condition  $\forall  u\in\D(\Delta)$ can \emph{equivalently} be replaced by the more restrictive condition  $\forall  u\in\D(\Delta)$ with $\Delta u\in\D(\E)$. Indeed, each $u\in\D(\Delta)$  can be approximated by $u_j:=P_{1/j}u$ in the latter set, and according to the proof of Lemma \ref{Hess-Lapl} (now with $u_j$ in the place of $u$), convergence $u_j\to u$ in $\D(\Delta)$ also implies convergence $|\nabla u_j|\to |\nabla u|$ in $\D(\E)$.
\end{remark}

\begin{remark}\label{rem-mBE} Given $(M,g)$, $N$, and $N$-admissible $k$, the following are equivalent:
\begin{itemize}
\item[(i)] $\forall u\in\D(\Delta):  |\nabla u|\in\D(\E)$ and 
 $\forall  \varphi\in\D_\loc(\E)\cap L^\infty_+$ with $|\nabla \varphi|\in L^\infty$: inequality \eqref{mBE} holds;
\item[(ii)]  $\forall u\in\D(\Delta)$ with $\Delta u\in\D(\E):  |\nabla u|\in\D(\E)$ and 
 $\forall  \varphi\in\D(\E)\cap L^\infty_+$ with $|\nabla \varphi|\in L^\infty$: inequality \eqref{mBE} holds;
 \item[(iii)]  $\forall u\in\D(\Delta)$ with $\Delta u\in\D(\E)$:  $v:=|\nabla u|\in\D(\E)$ and 
 $\forall  \varphi\in\D(\Delta)\cap L^\infty_+$ with $|\nabla \varphi|, \Delta \varphi\in L^\infty$: 
  \begin{align}\label{mwBE}\nonumber
\frac12 \int\Delta\varphi \,v^2dm &-\int\varphi\nabla u\nabla \Delta u \,dm+\int k^-\varphi v^2dm\\
 &\ge \int k^+\varphi v^2dm+
 \int\varphi|\nabla v|^2dm+\frac1{N-1}\int\varphi
\Big| |\nabla v|- |\Delta u|\Big|^2dm.
 \end{align}
\end{itemize}
\end{remark}
The latter formulation is close to the usual formulation of the Bakry-\'Emery condition (cf. \cite{AGS}, \cite{EKS}, \cite{Ketterer-cones}, \cite{Savare}).

\begin{definition} Let $(M,g)$, measurable $k:M\to\R$, and $N\in[1,\infty]$  be given.

 (i) We say that the \emph{Bakry-\'Emery inequality} $\BE_2(k,N)$ holds true 
  if \  
  $\forall u\in\D(\Delta)$ with $\Delta u\in\D(\E)$ and 
 $\forall  \varphi\in\D(\Delta)\cap L^\infty_+$ with $\Delta \varphi\in L^\infty$: 
  \begin{align}\label{wBE}\nonumber
\frac12 \int\Delta\varphi \,v^2dm &-\int\varphi\nabla u\nabla \Delta u \,dm+\int k^-\varphi |\nabla u|^2dm\\
 &\ge \int k^+\varphi |\nabla u|^2dm+
 \frac1{N}\int\varphi
 |\Delta u|^2dm.
 \end{align}

 (ii) We say that the \emph{Bakry-\'Emery inequality} $\BE_1(k,N)$ holds true 
 if \ $\forall u\in\D(\Delta)$ with $\Delta u\in\D(\E)$:  $v:=|\nabla u|\in\D(\E)$ and $\forall  \varphi\in\D(\Delta)\cap L^\infty_+$ with $\Delta \varphi\in L^\infty$: 
  \begin{align}\label{wBE}\nonumber
\frac12 \int\Delta\varphi \,v^2dm &-\int\varphi\nabla u\nabla \Delta u \,dm+\int k^-\varphi v^2dm\\
 &\ge \int k^+\varphi v^2dm+
 \int\varphi|\nabla v|^2dm+\frac1{N-1}\int\varphi
\Big| |\nabla v|- |\Delta u|\Big|^2dm.
 \end{align}
\end{definition}
\begin{lemma} 
Given $(M,g)$, $N\in[1,\infty]$, and $N$-admissible $k$, 
\begin{itemize}
\item[\bf (i) \ ] $\BE_1(k,N)\quad\Longrightarrow\quad \BE_2(k,N)$
\item[\bf (ii) \ ] $\BE_1(k,N)\quad\Longleftarrow\quad \BE_2(k,N) \ \& \ 
\BE_2(K,\infty)$ for some $K\in\R$

\item[\bf (iii) \ ] $\BE_1(k,N)\quad\Longrightarrow\quad \BE_1'(k,N)$
\item[\bf (iv) \ ] $\BE_1(k,N)\quad\Longleftarrow\quad \BE_1'(k,N)\ \& \ \BE_2(K,\infty)$ for some $K\in\R$.
\end{itemize}
\end{lemma}
\begin{proof}
\begin{itemize}
\item[\bf (i) \ ]  By Cauchy-Schwarz inequality,
$$ \int\varphi|\nabla v|^2dm+\frac1{N-1}\int\varphi
\Big| |\nabla v|- |\Delta u|\Big|^2dm\ge \frac1N \int (\Delta u)^2dm.$$
\item[\bf (ii) \ ] 
Formally, this follows along the lines of the proof of Corollary \ref{p-BE1}. Rigorously, it is derived in \cite{Sturm-BE}.
\item[\bf (iii) \ ] Obvious according to Remark \ref{rem-mBE}(iii).
\item[\bf (iv) \ ] To prove the claim, according to Remark \ref{rem-mBE}(iii) it suffices to get rid of the condition $|\nabla\varphi|\in L^\infty$ in the specification of the functions which satisfy \eqref{mwBE}.
Under the extra assumption of $\BE_2(K,\infty)$ for some $K\in\R$, the heat semigroup maps $L^\infty$ into functions with bounded gradients. More precisely,
$$\forall   \varphi\in\D(\Delta)\cap L^\infty_+\text{ with }\Delta \varphi\in L^\infty: \
\varphi_j:=P_{1/j} \in\D(\Delta)\cap L^\infty_+\text{ with }|\nabla \varphi|, \Delta \varphi\in L^\infty
$$
and 
$$\int \varphi_j wdm\to \int \varphi wdm, \quad \int \Delta\varphi_j wdm\to \int \Delta\varphi wdm\quad\text{as }j\to\infty\qquad(\forall w\in L^1).$$
\end{itemize}
\end{proof}

\begin{remark} If $(M,g)$ is smooth and $k$ is bounded, then the pointwise Bakry-\'Emery inequality implies all the other ones, in particular,
$$\pBE_2(k,N)\quad\Longrightarrow\quad \BE_2(k,N)\quad\Longrightarrow\quad \BE_1(k,N)\quad\Longrightarrow\quad \BE_1'(k,N).$$
\end{remark}


\subsection{A Fundamental Spectral Inequality}

Let $\text{\rm spec}\left(-\Delta\right)$ denote the  spectrum of $-\Delta$ and $\text{\rm spec}_{p}\left(-\Delta\right)$ the point spectrum. Denote by 
$\lambda_1:=\inf\big(\text{\rm spec}\left(-\Delta\right)\setminus\{0\}\big)$  
and 
$\lambda^{(p)}_1:=\inf\big(\text{\rm spec}_{p}\left(-\Delta\right)\setminus\{0\}\big)$ the 
respective spectral gaps. The major interest  lies in estimating $\lambda_1$.
Note however that $\lambda_1=\lambda^{(p)}_1$ whenever the essential spectrum $\text{\rm spec}_{ess}\left(-\Delta\right)$ is empty.

\begin{theorem}\label{l1-l0} Assume $\BE_1'(k,N)$.  Then 
\begin{equation}
\lambda^{(p)}_1\ge
\inf\text{\rm spec}\left(- \Delta+k\right).
\end{equation}
Moreover, if $N<\infty$ then
for every $t\in[0,1]$,
\begin{equation}\label{sixty}
\lambda^{(p)}_1\ge\alpha_{t}\cdot
\inf\text{\rm spec}\left(-t\frac{N}{N-1} \Delta+k\right)
\quad\text{with} \quad
\alpha_{t}:=\frac N{N-1}\bigg[1+\frac{\frac t{1-t}}{(N-1)^2}
\bigg]^{-1}.
\end{equation}
\end{theorem}
Obviously, $\alpha_t$ is decreasing in $t$, and
$$\alpha_t=\begin{cases}
\frac N{N-1}, \quad&\text{for }t=0\\
1, \quad&\text{for }t=\frac{N-1}N \\
0, \quad&\text{for }t=1.
\end{cases}$$

\begin{proof} 
Let $u$ be a non-constant function with $-\Delta f=\lambda u$ for some $\lambda>0$. 
Then  \eqref{mBE} implies
\begin{equation*}
\int \Big(\lambda |\nabla u|^2-k|\nabla u|^2-\big|\nabla|\nabla u|\big|^2\Big)dm\ge\frac{1-1/\delta}{N-1}
\int\big| \nabla|\nabla u|\big|^2dm+\frac{1-\delta}{N-1}\lambda
\int (\nabla u)^2dm
\end{equation*}
for every $\delta>0$. With $v:=|\nabla u|$ this reads
\begin{equation*}
\left(1-\frac{1-\delta}{N-1}\right)\lambda\ge \frac1{\int v^2dm}\int\left[\left(1+\frac{1-1/\delta}{N-1}\right)|\nabla v|^2+kv^2\right]dm
\end{equation*}
which implies, since $v\not\equiv 0$,
\begin{equation*}
\frac{N-2+\delta}{N-1}\,\lambda\ge
\inf\text{\rm spec}\left(-\frac{N-1/\delta}{N-1} \Delta+k\right).
\end{equation*}
Putting $t:=\frac{N-1/\delta}{N}$ yields the claim in the case $N<\infty$. In the case $N=\infty$, we can avoid the Cauchy-Schwarz argument and obtain directly $\lambda\ge
\inf\text{\rm spec}\left(- \Delta+k\right)$.
\end{proof}


 \section{Hardy Inequality}
\subsection{Hardy Weights}
For the sequel, we assume that $(M,g)$ is a (not necessarily smooth) complete Riemannian manifold.
The results of this subsection, however, will also hold without any essential changes for any infinitesimally Hilbertian metric measure space or any strongly local Dirichlet space. (In the latter case, the Sobolev space $W^{1,2}(M)$ should be replaced by the form domain $\D(\E)$.)

\begin{definition} We say that a lower bounded, measurable function $\vartheta$ is a \emph{Hardy weight} if
$-\Delta\ge \vartheta$
in the sense of
\begin{equation}\label{zwei}
\int_M |\nabla f|^2dm\ge\int_M f^2\,\vartheta dm\qquad\quad\forall f\in W^{1,2}(M).
\end{equation}
\end{definition}
Note that a Hardy weight in our sense is \emph{not} necessarily $\ge0$.

\begin{remark}
Given a (finite or countable) family $\{\vartheta_i\}_{i\in\mathfrak I}$ of Hardy weights and a family $\{s_i\}_{i\in\mathfrak I}$ of nonnegative numbers which add up to $s\le1$, also the function
$$\vartheta:=\sum_i s_i\vartheta_i$$ is a Hardy weight.
\end{remark}
\begin{lemma}[{\cite{Carron}, \cite{Fitz}}]\label{rho-ineq} 
 Every $\vartheta$ of the form $\vartheta=-\frac{\Delta\Phi}\Phi$  for some $\Phi>0$ is a Hardy weight.
Indeed, $\vartheta\le-\frac{\Delta\Phi}\Phi$ (in distributional sense) suffices.
\end{lemma}

\begin{remark}[{\cite{Li-Wang}}]
 Assume that there exists a positive Green function $G(x,y)$. Then for every   $z\in M$,
$\vartheta:=\frac14\left|\frac{\nabla G(z,.)}{G(z,.)}\right|$
is a Hardy weight.
\end{remark}

\begin{example}
On $\R^n$, for each $z$ the function
$$\vartheta(x)=\left(\frac{n-2}2\right)^2\, \frac1{|x-z|^2}$$
is a Hardy weight. Indeed, this can be concluded from any of the above criteria with 
$\Phi(x)={|x-z|^{-\frac{n-2}2}}$ and $G(z,x)={c_n}\,{|x-z|^{-(n-2)}}$, resp.
\end{example}

\subsection{Hardy Inequality on Manifolds}
Now assume that the complete Riemannian manifold $(M,g)$ is smooth with sectional curvature $\le \ell$ for some $\ell\in\R$ and with injectivity radius $\ge R_\ell$ where
$$ R_\ell:=\frac{\pi}{2\sqrt\ell}\text{ if }\ell>0, \qquad  R_\ell:=\infty \text{ if }\ell\le0.$$  
Without restriction, we may assume that $(M,g)$ has dimension $n\ge3$.
Put 
$$\underline\sin_\ell(r):=\sin_\ell(r\wedge R_\ell), \quad
\underline\cos_\ell(r):=\underline\sin'_\ell(r), \quad 
\underline\tan_\ell(r):=\frac{\underline\sin_\ell(r)}{\underline\cos_\ell(r)}$$
with $\sin_\ell(r)$ as defined in \eqref{sin-ell}.
More explicitly,  in the case $\ell>0$,  
$$\underline\sin_\ell(r):=\frac1{\sqrt\ell}\sin(\sqrt\ell r\wedge R_\ell), \quad
\underline\cos_\ell(r):=\cos(\sqrt\ell r\wedge R_\ell), \quad 
\underline\tan_\ell(r):=\frac1{\sqrt\ell}\tan(\sqrt\ell r\wedge R_\ell),$$
and in the case $\ell=0$,
$$\underline\sin_\ell(r):=r, \qquad
\underline\cos_\ell(r):=1, \qquad 
\underline\tan_\ell(r):=r.$$

\begin{theorem}\label{Hardy} For every $z\in\M$, 
\begin{equation}\label{eins}
-\Delta\ge \left(\frac{n-2}2\right)^2\, \frac{1}{\underline\tan^2_\ell\big(d(z,.)\big)}-\frac{n-2}2\ell\,{\bf 1}_{B_{R_\ell}(z)}.
\end{equation}
Or in other words,
\begin{equation}\label{zwei}
\int_M |\nabla f|^2dm\ge \left(\frac{n-2}2\right)^2\, \int_{B_{R_\ell}(z)} \frac{f^2(x)}{\tan^2_\ell\big(d(z,x)\big)}dm(x)-\frac{n-2}2\ell\int_{B_{R_\ell}(z)}f^2\,dm
\end{equation}
for all $f\in W^{1,2}(M)$.
\end{theorem}
\begin{proof} Let us first consider the case $\ell>0$.
By scaling, we may assume   $\ell=1$.  Put $B=B_{\pi/2}(z)$. 
To prove \eqref{zwei}, for given $f\in W^{1,2}(M)$ put $v:=f/\Phi$  with $\Phi$ as below in \eqref{drei}. Then
\begin{align*}
\int|\nabla f|^2&=\int |\nabla v|^2\Phi^2+\int v^2|\nabla\Phi|^2+\frac12\int \nabla v^2\nabla\Phi^2\\
&=\int |\nabla v|^2\Phi^2+\int v^2|\nabla\Phi|^2-\frac12\int v^2\Delta\Phi^2
\\
&\ge-\int_B v^2\Phi\Delta\Phi.
\end{align*}
By means of the subsequent Lemma, this yields
\begin{align*}
\int|\nabla f|^2&\ge\int_B v^2\Phi^2\left[ \left(\frac{n-2}2\right)^2\,  \frac{\cos_\ell\big(d(z,.)\big)}{\sin^2_\ell\big(d(z,.)\big)}-\frac{n-2}2\ell\right]\\
&= \left(\frac{n-2}2\right)^2\int_B f^2\,  \frac{\cos_\ell\big(d(z,.)\big)}{\sin^2_\ell\big(d(z,.)\big)}-\frac{n-2}2\ell\int_Bf^2
\end{align*}
which is the claim.
\end{proof}

\begin{lemma} \label{hard-lemm}
Put 
\begin{equation}\label{drei}\Phi(x):=\Big(\underline\sin_\ell\big(d(z,x)\big)\Big)^{-\frac{n-2}2}.
\end{equation} 
Then on $B:=B_{R_\ell}(z)$,
\begin{equation}\label{vier}-\frac{\Delta\Phi}\Phi\ge  \left(\frac{n-2}2\right)^2\, \frac{1}{\underline\tan^2_\ell\big(d(z,.)\big)}-\frac{n-2}2\ell.
\end{equation}
Moreover, the inequality \eqref{vier}  holds in distributional sense on all of $M$.
\end{lemma}

\begin{proof} Let us first treat the case $\ell>0$. Then by scaling, we may assume   $\ell=1$.
Laplace comparison (under upper bound $\ell=1$ for the sectional curvature and lower bound on the injectivity radius) for $\Phi=h(d(z,.))$ with
$h'(r)=-(n/2-1)\sin^{-n/2}(r)\cos(r)\le0$ yields
\begin{align*}
-\Delta\Phi&\ge\frac{-1}{\sin^{n-1}r}\,\frac d{dr}\Big(\sin^{n-1}(r)\,h'(r)\Big)\Big|_{r=d(z,.)}\\
&=(n/2-1)^2\,\sin^{-1-n/2}(r)\,\cos^2(r)-(n/2-1)\sin^{1-n/2}(r)\Big|_{r=d(z,.)}
\end{align*}
provided $r<\pi/2$.
Thus
\begin{align*}
\frac{-\Delta\Phi}\Phi&\ge 
  \left(\frac{n-2}2\right)^2\, \frac{\cos^2(r)}{\sin^2(r)}\bigg|_{r=d(z,.)}-\frac{n-2}2
\end{align*}
in $B_{R_\ell}(z)$.

In the case $\ell\le0$, formally the same calculations will prove the claim.
\end{proof}

\begin{remark} Assume $\ell>0$. For the  inequality \eqref{vier} to hold (in classical sense on the ball $B=B_{R_\ell}(z)$
 or in distributional sense on all of $M$), it suffices that the assumption on sectional curvature $\le \ell$ holds in a neighborhood of $\overline B$
 and that the assumption on injectivity radius $> R_\ell$ holds for $z$.
 \end{remark}
 
 \begin{corollary}\label{HaHardy}
 Let $(M,g)$ be as before and let a countable or finite set $\{z_i\}_{i\in\mathfrak I}\subset M$ be given. 
 
 (i) Then for all $s_i\ge0$ with  $\sum_{i\in\mathfrak I}s_i=1$,
 $$-\Delta\ge \left(\frac{n-2}2\right)^2\, 
   \sum_{i\in\mathfrak I} \frac{s_i}{\underline\tan^2_\ell\big(d(z_i,.)\big)}-\frac{n-2}2\ell\,\sum_i s_i{\bf 1}_{ B_{R_\ell}(z_i)}$$
 
 (ii) Assume $\ell>0$ and that $d(z_i,z_j)\ge 2R_\ell$ for all $i\not= j$. Then
 $$-\Delta\ge \left(\frac{n-2}2\right)^2\,  \sum_{i\in\mathfrak I} \frac{1}{\underline\tan^2_\ell\big(d(z,.)\big)}-\frac{n-2}2\ell\,{\bf 1}_{\bigcup_i B_{R_\ell}(z)}$$
 \end{corollary}
 
 \begin{proof} (i) is an obvious consequence of the previous Theorem. For (ii),  choose 
 $$\Phi=\sum_i \Big(\underline\sin_\ell\big(d(z_i,x)\big)\Big)^{-\frac{n-2}2}.$$
 and follow the argumentation of the previous proof.
 \end{proof}
 

\begin{remark}
(i) The assertion of Theorem \ref{Hardy} remains true for any complete Riemannian manifold $(M,g)$ with non-smooth metric and any number $\ell\in\R$ provided
\begin{equation}\label{ubc}\Delta d(.,z)\ge (n-1)\frac1{\tan_\ell d(.,z)} \qquad \text{in a neighborhood of } B_{R_\ell}(z).\end{equation}

(ii) By Laplace comparison theorem, the latter holds if $(M,g)$ is smooth with sectional curvature $\le \ell$ for some $\ell\in\R$ and with injectivity radius $> R_\ell$.

(iii) More generally, \eqref{ubc} holds if $(M,g)$ is a ${\sf CAT}(\ell)$-space.
Indeed, for such a space with synthetic upper bounds for the sectional curvature in the sense of Alexandrov,  a Hessian comparison theorem holds
which in turn implies a Laplace comparison theorem of the requested form.
\end{remark}

\subsection{Hardy Estimate for  Conical Singularities}
If $M$ has a conical singularity, then proving a Hardy estimate at the singularity  does not require an upper bound on the sectional curvature but merely
a bound on the curvature in radial direction which in turn can easily be provided by the choice of the warping function.

\begin{theorem} \label{ha-co}
Assume that $M$ has a conical singularity at $z$ such that
$$M\supset 
B_{\rho}(z)\simeq [0,\rho)\times_{f} N$$ 
with a complete $(n-1)$-dimensional Riemannian manifold $N$,
with 
 $f(r)=\sin_{\ell}(r)$ for some   $\ell\in\R$, and with $\rho\in(0, R_\ell]$.
Then for every $L\ge\ell\vee (\frac\pi{2\rho})^2>0$,
$$-\Delta\ge 
 \left(\frac{n-2}2\right)^2\, \frac{1}{\underline\tan_{L}^2(d(z_i,.))}-\frac{n-2}2L {\bf 1}_{B_{R_L}(z)}.$$
\end{theorem}
Note that no assumption on $N$ and no assumption on $M\setminus B_\rho(z)$ is made (besides completeness).

\begin{proof} Choosing
\begin{equation*}\Phi(x):=\underline\sin^{1-n/2}_L\big(d(z,x)\big).
\end{equation*} 
and following the argumentation in the proof of Theorem \ref{Hardy}, we obtain for $r<R_L$,
\begin{align*}
-\Delta\Phi&=
-\Big(\underline\sin^{1-n/2}_L\Big)''(r)-\frac{n-1}{\tan_\ell(r)}\Big(\underline\sin^{1-n/2}_L\Big)'(r)\\
&\ge
-\Big(\underline\sin^{1-n/2}_L\Big)''(r)-\frac{n-1}{\tan_L(r)}\Big(\underline\sin^{1-n/2}_L\Big)'(r)\\
&=(n/2-1)^2\,\sin_L^{-1-n/2}(r)\,\cos_L^2(r)-(n/2-1)L\sin_L^{1-n/2}(r).
\end{align*}
Thus
$$-\Delta\ge\frac{-\Delta\Phi}\Phi\ge  \left(\frac{n-2}2\right)^2\, \frac{1}{\underline\tan_{L}^2(d(z_i,.))}-\frac{n-2}2L.$$
\end{proof}

\begin{corollary} Assume that $M$ has 
a finite or countable
set of singularities 
$\{z_i\}_{i\in\mathfrak I}\subset M$  such that for each $i$
$$M\supset 
B_{\rho_i}(z_i)\simeq [0,\rho_i)\times_{f_i} N_i$$ 
with  complete $(n-1)$-dimensional Riemannian manifolds $N_i$,
with 
 $f_i(r)=\sin_{\ell_i}(r)$ for some   $\ell_i\in\R$, and with $\rho_i\in(0, R_{\ell_i}]$.
 Assume further that $d(z_i,z_j)\ge \rho_i+\rho_j$ for all $i\not=j$.
Then for every choice of $L_i\ge\ell_i\vee (\frac\pi{2\rho_i})^2>0$,
$$-\Delta\ge 
 \left(\frac{n-2}2\right)^2\, \sum_i\frac{1}{\underline\tan_{L_i}^2(d(z_i,.))}-\frac{n-2}2\sum_iL_i {\bf 1}_{B_{L_i}(z_i)}$$
\end{corollary}

\subsection{Spectral Gap Estimate Involving Hardy Weights}

Let $(M,g)$ be a complete, not necessarily smooth Riemannian manifold of dimension $n\ge3$.

\begin{lemma}  Assume $\BE_1'(k,n)$ holds and that 
$k\ge K-\vartheta$
with a number $K$ and a Hardy weight $\vartheta$. Then
\begin{equation*}
\lambda^{(p)}_1\ge K.
\end{equation*}
More generally, assume that
$k\ge K-s\vartheta$
for some $s\in(0,\frac n{n-1})$. Then 
\begin{equation}
\lambda^{(p)}_1\ge \bar\alpha_s \cdot K
\end{equation}
with $$\bar\alpha_s:=\frac{n+\frac{s}{1-s}}{n-1+\frac{s}{1-s}}=1+\left(n-1+\frac{s}{1-s}\right)^{-1}.
$$
\end{lemma}

\begin{proof} Theorem \ref{l1-l0} with $N$ replaced by $n$ and with
$\bar\alpha_s=\alpha_t$ for $t={s\frac{n-1}{n}}$.
\end{proof}

%
%


\begin{theorem}\label{gap-smooth} Assume that
$(M,g)$ satisfies \eqref{ubc} with $\ell\in\R$ 
and $\BE_1'(k,n)$ 
with
\begin{equation*}k\ge K-\frac{(n-2)^2}4\sum_i\frac{s_i}{\tan^2_\ell d(z_i,.)}+
\frac{n-2}2\ell\sum_is_i {\bf 1}_{B_{L_i}(z_i)}
\end{equation*}
with (finitely or countably many) points $z_i\in M$ and 
with  nonnegative numbers $s_i$ satisfying $\sum_i s_i=:s< {\frac n{n-1}}$. 
Then
\begin{equation}
\lambda^{(p)}_1\ge 
\bar\alpha_s \cdot K.
\end{equation}

Moreover, if 
$d(z_i,z_j)\ge 2R_\ell$ for all $i\not= j$,
 then the previous holds true with $s:=\sup_i s_i$ in the place of $s:=\sum_i s_i$.
\end{theorem}

\begin{proof}
According to Corollary \ref{HaHardy}, we may apply  the previous Lemma with
$$\vartheta:=\frac1s 
\bigg[\frac{(n-2)^2}4\sum_i\frac{s_i}{\tan^2_\ell d(z_i,.)}
-\frac{n-2}2\ell\sum_is_i {\bf 1}_{B_{L_i}(z_i)}\bigg].$$
%
\end{proof}

\section{Manifolds with Singularities}

\subsection{Bakry-\'Emery for Admissible Manifolds}
\begin{definition}
A $n$-dimensional Riemannian manifold $(M,g)$  is called \emph{admissible} if 
$g$ is smooth on $M\setminus M_0$ and
 $$\Ric^M\ge k\quad\text{on }M\setminus M_0$$
 for some closed $m$-zero set $M_0\subset M$ and some $n$-admissible function
$k:M\to\R$.
%
%
\end{definition}

\begin{theorem}\label{BE-00} Assume that there exist $M_0\subset M$ and $k:M\to\R$ such that
\begin{itemize}
\item $\Ric\ge k$ on $M\setminus M_0$,
\item $k$ is $n$-admissible,
\item $\C^\infty_c(M\setminus M_0)$ is dense in $\D(\Delta)$.
\end{itemize}
Then    $\BE_1'(k,n)$ holds true.
\end{theorem}

\begin{proof}  
If we assume $u\in \D_*:=\C^\infty_c(M\setminus M_0)$ then $|\nabla u|\in\D(\E)$, and according to 
Corollary \ref{p-BE1},
the (`self-improved') Bochner inequality  holds pointwise on $M_0$, that is,
\begin{equation*}  \big|\Delta u\big|^2(x)\ge  \big|\nabla |\nabla u|\big|^2(x)+k(x) |\nabla u|^2(x)+\frac1{n-1}
\Big| \big|\nabla |\nabla u|\big|- \big|\Delta u\big|\Big|^2(x)\qquad\forall x\in M. \end{equation*}
Multiplying by $\varphi$, integrating w.r.t.~$dm$, and performing integration by parts proves the estimate $\BE_1'(k,n)$ for  $u\in \D_*$.

By assumption,
$\D_*$ is dense in $\D(\Delta)$. Thus for each $u\in\D(\Delta)$ there exist $u_n\in\D_*$ such that $u_n\to u$ in $\D(\Delta)$.
In particular, $|\nabla u_n|\to |\nabla u|$ in $L^2$ and thus by lower semicontinuity of $\E$ on $L^2$ and by validity of the estimate \eqref{Hess--Laplace} for $u_n\in\D_*$,
\begin{align*}
\int\big|\nabla |\nabla u|\big|^2dm&\le \liminf_n \int\big|\nabla |\nabla u_n|\big|^2dm\\
&\le \liminf_n \left[
C\int \big|\Delta u_n\big|^2dm+C'\int\big|\nabla u_n\big|^2dm\right]\\
&=C\int \big|\Delta u\big|^2dm+C'\int\big|\nabla u\big|^2dm<\infty.
\end{align*}
Hence, $|\nabla u|\in\D(\E)$ and $|\nabla u_n|\to |\nabla u|$ in $\D(\E)$.

This finally proves that each term in \eqref{mBE} is continuous w.r.t. convergence in $\D(\Delta)$, and the estimate for all $u\in\D(\Delta)$ follows by its validity for $u\in\D_*$ and by the requested density of $\D_*$ in $\D(\Delta)$.
\end{proof}

\subsection{Bakry-\'Emery for Manifolds with Conical Singularities}
\begin{proposition}  \label{sing-adm}
Assume that $n\ge 3$ , that $g$ is smooth on $M\setminus M_0$ with
discrete $M_0:=\{z_i\}_i\subset M$, and that
for each $i$,
$$M\supset M_i:=B_{\rho_i}(z_i)\simeq [0,\rho_i)\times_{f_i} N_i$$ with 
$f_i(r)=\sin_{\ell_i}(r)$ for some   $\ell_i\in\R$, with 
$\rho_i\in(0,R_{\ell_i}]$, and with a complete $(n-1)$-dimensional Riemannian manifold $N_i$ with
 $$\Ric^{N_i}\ge (n-2){\kappa_i}$$ for some $\kappa_i\le1$.
 Assume further that 
 $\Ric^M\ge K$ on $M\setminus \bigcup_iM_i$, and
 $\rho_i\ge\rho>0$,  $\ell_i\ge\ell$ ($\forall i$) for some $K,\rho, \ell\in\R$.

 (i)  Then there exists $n$-admissible $k$ such that $\Ric^M\ge k$ on $M\setminus M_0$ if and only if 
  \begin{equation}
  \label{super-crit}
  \kappa_i>\frac{n(6-n)-4}{4(n-1)}\qquad(\forall i).
  \end{equation}
  
   (ii) The set $\C^\infty_c(M\setminus M_0)$ is dense in $\D(\Delta)$ 
  if and only if  $n\ge4$.
\end{proposition}

\begin{proof} 
(i) 
Assume $n>2$. Given $\epsilon>0$, choose $\alpha\in(0,\frac{n-2}2)$ such that $\alpha(n-2-\alpha)\ge(1-\epsilon)\left(\frac{n-2}2\right)^2$.
For each $i$, put $\lambda_i:=\ell_i\vee(\frac\pi{2\rho_i})^2>0$ and
$$
\Phi_i(x):=\underline\sin^{-\alpha}_{\lambda_i}\big(d(z_i,x)\big).$$
Then $\Phi_i\in\D(\E)$ due to our choice of $\alpha$, and $\Phi_i-1$ is supported in $M_i':=B_{\pi/(2\sqrt\lambda)}$. Without restriction, we may assume 
 that the $M_i'$ are pairwise disjoint.
Put
$$\Phi:=1+\sum_i(\Phi_i-1).$$
 Then
 $\Phi=\Phi_i$ on $M_i'$ and $\Phi=1$ on $M\setminus \bigcup M_i'$.

Similar as in the proof of Lemma \ref{hard-lemm}, now using the fact that $\tan_\lambda(r)\ge \tan_\ell(r)$, 
we conclude for $x\in M_i'$ and with  $r:=d(x,z_i)$,
\begin{align*}
-\Delta\Phi_i(x)&
=-\partial_r^2 \Phi_i(r)-\frac{n-1}{\tan_{\ell_i}(r)}\partial\Phi_i(r)\\
&\ge
-\partial_r^2 \Phi_i(r)-\frac{n-1}{\tan_{\lambda_i}(r)}\partial\Phi_i(r)\\
&=\alpha(n-2-\alpha)\,\sin_{\lambda_i}^{-\alpha-2}(r)\,\cos_{\lambda_i}^2(r)-\alpha\sin_{\lambda_i}^{-\alpha}(r).
\end{align*}
Thus
\begin{align*}
\frac{-\Delta\Phi_i}{\Phi_i} (x)&\ge 
 (1-\epsilon) \left(\frac{n-2}2\right)^2\, \frac{1}{\tan_{\lambda_i}^2(r)}-\frac{n-2}2\lambda_i.
\end{align*}

Put
$$\vartheta:=\sum_i\left[(1-\epsilon) \left(\frac{n-2}2\right)^2\, \frac{1}{\tan_{\lambda_i}^2(d(z_i,.))}-\frac{n-2}2\lambda_i
\right]\, 1_{B_i'}.$$
Then 
$$\vartheta\le-\frac{\Delta\Phi}\Phi=\begin{cases}-\frac{\Delta\Phi_i}{\Phi_i},\quad&\text{on }M'_i\\
0,\quad&\text{on }M\setminus \bigcup_i M_i'
\end{cases}$$
and thus, according to Remark \ref{rho-ineq},
\begin{equation*}
-\Delta\ge \vartheta\quad\text{on }M.
\end{equation*}
Since $\big|\frac{1}{\tan_{\lambda_i}^2(d(z_i,.))}-\frac1{d^2(z_i,.)}\big|\le C$ for some $C$ and all $i, x$ under consideration, there exists a constant $C_0$, independent of $\epsilon$,
such that
\begin{equation*}
-\Delta\ge (1-\epsilon) \left(\frac{n-2}2\right)^2\, \sum_i\frac{1}{d^2(z_i,.)}-C_0.\quad\text{on }M.
\end{equation*}
In this estimate, we obviously may pass to the limit $\epsilon\to0$ and thus obtain 
\begin{equation}\label{uno}
-\Delta\ge  \left(\frac{n-2}2\right)^2\, \sum_i\frac{1}{d^2(z_i,.)}-C_0.\quad\text{on }M.
\end{equation}

Now let us have a look on lower Ricci bunds.
By assumption $\Ric^M\ge K$ on $M\setminus \bigcup_iM_i$. Hence, $\Ric^M\ge K'$ on $M\setminus \bigcup_iM'_i$ for some  $K'\in\R$. Furthermore, for $x\in B_i'$, according to \eqref{ric-cone}
$$\Ric_x^M\ge k_i(x):=(n-1)\ell_i -(n-2)\frac{(1-\kappa_i)^+}{\sin^2_{\ell_i}(d(x,z_i))}.$$
Then on $M\setminus\{z_i\}$,
$$\Ric^M\ge k:=1_{M\setminus\bigcup_i M_i'}K'+\sum_i k_i1_{M_i'}.$$
Since $\big|\frac{1}{\sin_{\ell_i}^2(d(z_i,.))}-\frac1{d^2(z_i,.)}\big|\le C$ for some $C$ and all $i, x$ under consideration, there exists a constant $C_1$
such that
\begin{equation}\label{duo}
k\ge -C_1-(n-2)\sum_i\frac{(1-\kappa_i)^+}{d^2(x,z_i))}.
\end{equation}
Comparing \eqref{duo} with \eqref{uno}, we see that 
$k^-$ is form bounded w.r.t. $-\Delta$ with form bound $<\frac n{n-1}$  
if
$$(n-2)(1-\kappa_i)^+<\frac n{n-1}\left(\frac{n-2}2\right)^2\qquad(\forall i).$$
Indeed, the latter is also necessary as can be easily verified. 
%

(ii) Let us first assume  for simplicity that there is only one singularity and that $M$ is given as the warped product $[0,\rho)\times_{f} N$.
Then according to Ketterer's work \cite[pf.~of Thm.~3.12]{Ketterer-cones}, the Laplacian $\Delta$ on $M$, restricted to $\D_z:=\C^\infty_c(M\setminus\{z\})$, is essentially self-adjoint if (and only if) $n\ge4$.
Essential self-adjointness of course implies density of $\D_z$ in $\D(\Delta)$.
By localization, this argument carries over to the general case of a discrete set $M_0$.
\end{proof}

\begin{theorem} \label{thm-be}
Assume that 
$g$ is smooth on $M\setminus M_0$ with
discrete $M_0:=\{z_i\}_i\subset M$, 
that $$\Ric^M\ge k\quad\text{on }  M\setminus  \{z_i\}_i,$$
and that
for each $i$,
$$M\supset M_i:=B_{\rho_i}(z_i)\simeq [0,\rho_i)\times_{f_i} N_i$$ with  $f_i(r)=\sin_{\ell_i}(r)$ for some   $\ell_i\in\R$, with 
$\rho_i\in(0,R_{\ell_i}]$, and with a complete $(n-1)$-dimensional Riemannian manifold $N_i$ with finite volume and
 $$\Ric^{N_i}\ge (n-2){\kappa_i}$$ for some $\kappa_i\in\R$, 
and  that 
 $\Ric^M\ge K$ on $M\setminus \bigcup_iM_i$, and
 $\rho_i\ge\rho>0$,  $\ell_i\ge\ell$ ($\forall i$) for some $K,\rho, \ell\in\R$.
%
%
%
 Without restriction, we  may assume
  $$k\ge K\,1_{M\setminus \bigcup_i M_i}+ \sum_i\bigg[(n-1)\ell_i -(n-2)\frac{(1-\kappa_i)^+}{\sin^2_{\ell_i}(d(x,z_i))}\bigg]\,1_{M_i}.$$
  Assume further
  $n\ge 3$ and  $\Ric^{N_i}\ge (n-2){\kappa_i}$ with
   \begin{equation}
  \label{super-crit}
\kappa_i>\frac{n(6-n)-4}{4(n-1)}\qquad(\forall i)
  \end{equation}
  or $n=2$ and 
  $\diam(N_i)\le\pi \ (\forall i)$.
Then 
$k$ is $n$-admissible  and
$\BE'_1(k,n)$ is satisfied.
  \end{theorem}

Criterion \eqref{super-crit} in particular is fulfilled if
  \begin{itemize}
 \item[--] $n\ge 6$ and  $\kappa_i\ge0$ arbitrary;
 \item[--] $n=5$ and $\kappa_i>\frac1{16}$;
 \item[--] $n=4$ and $\kappa_i>\frac13$;
 \item[--] $n=3$ and $\kappa_i>\frac58$.
 \end{itemize}

%
 
 \begin{proof} (a) It remains to prove that the $\BE_1'(k,n)$ inequality \eqref{mBE} holds true  for all $u\in\D(\Delta)$ and for all $\varphi\in\D_\loc(\E)\cap L^\infty_+$ with $|\nabla \varphi|\in L^\infty$. Given such a $\varphi$, define 
  $\varphi_i\in\D_\loc(\E)\cap L^\infty_+$ with $|\nabla \varphi_i|\in L^\infty$ and $\supp[\varphi]\subset M_i$
  for $i\in\mathfrak I\cup\{0\}$ by
  $\varphi_i:=\varphi \chi_i$
 with standard cut-off functions
 $\chi_i:=1\wedge\big[2- \frac2{\rho_i}d(z_i,.)\big]\vee0$ for $i\ge1$ and 
 $\chi_0:=1-\sum_i\chi_i$.
 Here $M_0:=M\setminus\bigcup_i M_i$.
 Since  \eqref{mBE}  is linear in $\varphi$ and since $\varphi=\sum_{i\in\mathfrak I\cup\{0\}}\varphi_i$, it suffices to prove  \eqref{mBE} with $\varphi_i$ for $i=0$ and for $i\in\mathfrak I$ in the place of $\varphi$.
 
 (b) Let us first consider the case $\varphi=\varphi_0$ with $\supp[\varphi]\subset M_0$.
 Thus all singularities of $M$ are out of sight
 and $M_0$ could equally well regarded as an open subset of a  \emph{smooth} complete Riemannian manifold $\tilde M$ with uniform lower Ricci bound. (For instance, the conical singularities could be modified into cylindrical ends.)
 Since the Laplacian is a local operator,
 $$\Big\{u\big|_{M_0}: u\in \D(\Delta^M)\Big\}=\Big\{u\big|_{M_0}: u\in \D(\Delta^{\tilde M})\Big\}.$$
 On $\tilde M$, the Laplacian restricted to $\C_c^\infty(\tilde M)$ is essentially self-adjoint and thus
 $\C_c^\infty(\tilde M)$ is dense in $\D(\Delta^{\tilde M})$. Hence, 
 $$\C^\infty(M_0)\cap \D(\Delta) \text{ is dense in }\D(\Delta).$$
 For $u\in \C^\infty(M_0)\cap \D(\Delta)$, the pointwise $\BE_1(k,n)$ inequality \eqref{elf} holds true. Multiplying this by $\varphi$, integrating it, and performing integration by parts, the 
 $\BE_1'(k,n)$ inequality \eqref{mBE} follows.
Since $k^-$ is form bounded with sufficiently small form bound, all terms in \eqref{mBE} are continuous w.r.t.~$u\in\D(\Delta)$. Thus with the previous density assertion we conclude that \eqref{mBE} holds for all $u\in\D(\Delta)$.
 
 (c) Now assume that $\varphi=\varphi_i$  and thus $\supp[\varphi]\subset M_i$ for some $i\in\mathfrak I$.
 Then $M_i$ can equally well regarded as the ball of radius $\rho_i$ around the tip of the cone
 $$\tilde M_i:=I_i\times_{f_i} N_i$$
 with $I_i=[0,\pi/\sqrt{\ell_i}]$ if $\ell_i>0$ and $I_i=[0,\rho_i]$ else. Assuming this for convenience, 
 according to   \cite[Thm.~3.12]{Ketterer-cones}, the set 
$$\Xi_i:=\Big[\D(\Delta^{I_i})\otimes E_{i,0}\Big]\oplus \sum_{j=1}^\infty \C_c^\infty(\mathring I_i)\otimes E_{i,j}$$
is dense in $\D(\Delta^{\tilde M_i})$
provided 
\begin{equation}\label{crit}
\left(\frac{n-2}2\right)^2+\lambda_1^{N_i}\ge 1\end{equation}
Here $-\Delta^{I_i}$ denotes the nonnegative self-adjoint operator on $L^2(I_i,\sin_{\ell_i}^{n-1}(r)dr)$ associated with the Dirichlet form
$$\E^{I_i}(v)=\int_{I_i} |v'|^2(r) \sin_{\ell_i}^{n-1}(r)dr),$$ 
$E_{i,j}\subset L^2(N_i)$ for $j\in\N\cup\{0\}$ denotes the eigenspace corresponding to the $j$-th eigenvalue of $-\Delta^{N_i}$, and 
$\lambda_1^{N_i}$ denotes the spectral gap of $N_i$.
(Ketterer assumed $\lambda_1^{N_i}\ge n-1$ but his argumentation works whenever \eqref{crit} is satisfied.)

In the case $n\ge 3$, since $\lambda_1^{N_i}\ge (n-1)\kappa$, criterion \eqref{crit} follows from
$$\left(\frac{n-2}2\right)^2+(n-1)\kappa_i\ge 1$$
which in turn follows from \eqref{super-crit}. In the case $n=2$, 
 \eqref{crit}  is equivalent to $\diam(N_i)\le \pi$.
 
(d) For  $\varphi=\varphi_i$ as above and functions $u\in  \sum_{j=1}^\infty \C_c^\infty(\mathring I_i)\otimes E_{i,j}$,  the $\BE'_1(k,n)$-inequality \eqref{mBE} will be deduced following the argumentation in \cite{Ketterer-cones}.

(e) Now consider $u\in \D(\Delta^{I_i})\otimes E_{i,0}$, say $u=v\otimes \mathbf 1$ with $v\in \D(\Delta^{I_i})$, and put
$\Phi(r):=\int_{N_i} \varphi(r,\xi)d\vol^{N_i}(\xi)$. Then 
$\Phi\in\D(\E^{I_i})$ and
$$\int_{\tilde M_i}\varphi u\,dm=\int_{I_i} \Phi v\sin_{\ell_i}^{n-1}dr, \qquad
\int_{\tilde M_i}\nabla\varphi \nabla u\,dm=\int_{I_i} \Phi' v' \sin_{\ell_i}^{n-1}dr
$$
as well as 
$\Gamma^{\tilde M_i}(u)=\Gamma^{I_i}(v)\otimes \mathbf 1$ and $\Delta^{\tilde M_i}u=\Delta^{I_i}v\otimes \mathbf 1$. Therefore, $(u,\varphi)$ satisfies  the  $\BE'_1(K_i,n)$-inequality  for $\E^{\tilde M_i}$ if and only if $(v,\Phi)$ satisfies  the $\BE'_1(K_i,n)$-inequality  for $\E^{I_i}$. The latter in turn is always true with $K_i:=(n-1)\ell_i$ as a consequence of the implications
$\BE_2(K_i,n) \Rightarrow \BE_1(K_i,n) \Rightarrow  \BE'_1(K_i,n)$
and the well-known fact that $\BE_2(K_i,n)$  holds true for $\E^{I_i}$.

(f) Finally, we have to consider $\varphi=\varphi_i$ as above and $u=u_0+u_1$ with $u_0\in \D(\Delta^{I_i})\otimes E_{i,0}$ and $u_1\in  \sum_{j=1}^\infty \C_c^\infty(\mathring {I_i})\otimes E_{i,j}$. Then there exists a further decomposition $\varphi_i=\varphi_{i0}+\varphi_{i1}$ with $u_1\equiv0$ on $\supp[\varphi_{i0}]$ and $z_i\not\in \supp[\varphi_{i1}]$.
Then $(u_0+u_1,\varphi_{i0})$ satisfies \eqref{mBE} if and only if $(u_0,\varphi_{i0})$ satisfies \eqref{mBE}, and the latter follows form the discussion in (e) above.
Moreover, $(u_0+u_1,\varphi_{i1})$ satisfies \eqref{mBE} according to the previous discussion in (b).
 \end{proof}

\subsection{Spectral Gap Estimates}
Now we apply our results on Bakry-\'Emery and Hardy inequalities on manifolds with conical singularities to estimate the point spectral gap $\lambda^{(p)}_1$.
The question whether $\lambda_1=\lambda^{(p)}_1$ will be addressed in the next section.

Assume that $(M,g)$ has a conical singularity at $z\in M$ such that
  \begin{itemize}
\item
$\Ric^M\ge K$ on $M_0:=M\setminus  B_{\rho}(z)$
\item $M_1:=B_{\rho}(z)\simeq [0,\rho)\times_{f_{\ell}} N$  and $\Ric^{N}\ge (n-2)\kappa$.
\end{itemize}
with positive numbers $K,\rho,\ell\in\R$ with $\rho\le R_{\ell}$,  and an $(n-1)$-dimensional manifold $N$ with $\vol(N)<\infty$.

\begin{theorem} Assume that 
$n\ge 3$
and 
$\kappa>\frac{n(6-n)-4}{4(n-1)}$, cf. \eqref{super-crit}. Then
\begin{align}\lambda_1^{(p)}&\ge
\alpha\min\bigg\{K, (n-1)\ell-
n(1-\kappa)^+\Big(\frac{\pi}{2\rho}\Big)^2
\bigg\}.
\end{align}
with $$
t:={(1-\kappa)^+}\frac{4(n-1)}{n(n-2)}\in[0,1), \qquad
\alpha:=
\frac n{n-1}\bigg[1+\frac{\frac t{1-t}}{(n-1)^2}
\bigg]^{-1}.$$
In particular, 
\begin{itemize}
\item[(i)] if $K\le (n-1)\ell$ then
\begin{align*}\lambda_1^{(p)}&\ge
\alpha\, \Big[ K-
n(1-\kappa)^+\Big(\frac{\pi}{2\rho}\Big)^2
\Big]
\end{align*}
\item[(ii)]  if ${(1-\kappa)^+}\le \frac{n-2}4$ then $\alpha\ge1$.
\end{itemize}

\end{theorem}

\begin{proof} Put $L:=\frac{\pi^2}{4\rho^2}\ge \ell$ 
and
\begin{align*}k:=&\begin{cases}
K,\quad &\text{in } M_0\\
(n-1)\ell-(n-2)\frac{(1-\kappa)^+}{\sin_\ell^2d(z,.)},\quad \ &\text{in } M_1,
\end{cases}\\
\vartheta:=&\begin{cases}
0,\quad &\text{in } M_0\\
\Big(\frac{n-2}2\Big)^2\frac1{\underline\tan^2_{L}d(z,.)}-\frac{n-2}2L\qquad &\text{in } M_1
\end{cases}
\end{align*}
Then $\Ric^M\ge k$ on $M$ and $-\Delta\ge \vartheta$. 
According to Proposition \ref{sing-adm}, $k$ is $n$-admissible, and according to Theorem \ref{thm-be} the mild Bakry-\'Emery estimate $\BE'_1(k,n)$ holds. According to Theorem \ref{ha-co}, $\vartheta$ is a Hardy weight. 
Thus according to Theorem \ref{l1-l0},
\begin{align*}\lambda^{(p)}_1
&\ge \lambda_0(-\Delta+k)\ge \inf_M(\vartheta+k)\\
&=\min\bigg\{K, \inf_{M_1}\bigg[ (n-1)\ell-(n-2)\frac{(1-\kappa)^+}{\sin_\ell^2d(z,.)}+\Big(\frac{n-2}2\Big)^2\frac1{\tan^2_{L}d(z,.)}-\frac{n-2}2L\bigg]\bigg\}.
\end{align*}
More generally,
for every $t\in[0,1]$,
\begin{align*}\lambda_1^{(p)}&\ge\alpha_t \lambda_0\Big(-\frac{tn}{n-1}\Delta+k\Big)\ge\alpha_t \inf_M\Big(\frac{tn}{n-1}\vartheta+k\Big)\\
&=\alpha_t\min\bigg\{K, (n-1)\ell-\frac{tn}{n-1}\frac{n-2}2L\\
&\qquad\qquad\quad+\inf_{M_1}\bigg[ -(n-2)\frac{(1-\kappa)^+}{\sin_\ell^2d(z,.)}+\frac{tn}{n-1}\Big(\frac{n-2}2\Big)^2\frac1{\tan^2_{L}d(z,.)}\bigg]\bigg\}\\
&=\alpha_t\min\bigg\{K, (n-1)\ell-\frac{tn}{n-1}\frac{n-2}2L-(n-2)(1-\kappa)^+\ell\\
&\qquad\qquad\quad+\inf_{M_1}\bigg[ -(n-2)\frac{(1-\kappa)^+}{\tan_\ell^2d(z,.)}+\frac{tn}{n-1}\Big(\frac{n-2}2\Big)^2\frac1{\tan^2_{L}d(z,.)}\bigg]\bigg\}\\
&=\alpha_t\min\bigg\{K, (n-1)\ell-\frac{tn}{n-1}\frac{n-2}2L-(n-2)(1-\kappa)^+\ell\\
&\qquad\qquad\quad+\inf_{M_1}\bigg[
\bigg( -(n-2){(1-\kappa)^+}+\frac{tn}{n-1}\Big(\frac{n-2}2\Big)^2\bigg)\frac1{\tan^2_{L}d(z,.)}\\
&\qquad\qquad\qquad\qquad-(n-2){(1-\kappa)^+}\bigg(\frac1
{\tan_\ell^2d(z,.)}-\frac1{\tan^2_{L}d(z,.)}\bigg)\bigg]\bigg\}.
\end{align*}
Choosing $t={(1-\kappa)^+}\frac{4(n-1)}{n(n-2)}\in [0,1)$,  the pre-factor in the second last line vanishes
$$ -(n-2){(1-\kappa)^+}+\frac{tn}{n-1}\Big(\frac{n-2}2\Big)^2=0.$$
Furthermore, since $0\le\ell\le L$ and $\sqrt L d(z,.)\le\frac\pi2$,
$$0\le \frac1{\tan^2_{\ell}d(z,.)}-\frac1{\tan^2_{L}d(z,.)}\le \frac23(L-\ell)$$
according to  Lemma \ref{Euler} which we postposed to the  end of this section. Therefore,
\begin{align*}\lambda_1^{(p)}&\ge
\alpha_t\min\bigg\{K, (n-1)\ell-
2(1-\kappa)^+L
-(n-2)(1-\kappa)^+\ell\\
&\qquad\qquad\qquad\qquad\qquad
-\frac23(n-2){(1-\kappa)^+}(L-\ell)\bigg\}\\
&=
\alpha_t\min\bigg\{K, (n-1)\ell-
\frac13(1-\kappa)^+\Big[2(n+1)L+(n-2)\ell\Big]
\bigg\}\\
&\ge
\alpha_t\min\bigg\{K, (n-1)\ell-
(1-\kappa)^+nL
\bigg\}.
\end{align*}
\end{proof}

\begin{corollary} Assume in addition 
$\rho=R_\ell$
and $\kappa\le 1$. Then with $\alpha=\alpha_t$ as above,
\begin{align}\nonumber\lambda_1^{(p)}&\ge
\alpha_t\min\Big\{K, \ \big(n\kappa-1\big)\ell\Big\}.
\end{align}
\end{corollary}

The previous Theorem easily extends to manifolds with multiple singularities.
Assume that $(M,g)$ is smooth outside a discrete 
set of singularities $\{z_i\}_{i\in\mathfrak I}\subset M$, that there exist positive numbers $K,\rho_i,\ell_i\in\R$ with 
$\rho_i\le R_{\ell_i}$,  and $(n-1)$-dimensional manifolds $N_i$ with $\vol(N_i)<\infty$ such that \begin{itemize}
\item
$\Ric^M\ge K$ on $M_0:=M\setminus \bigcup_{i\in\mathfrak I} B_{\rho_i}(z_i)$
\item $M_i:=B_{\rho_i}(z_i)\simeq [0,\rho_i)\times_{f_{\ell_i}} N_i$  and $\Ric^{N_i}\ge (n-2)\kappa_i$.
\end{itemize}
Moreover, assume that the sets $M_i$  for $i\in \mathfrak I$ are pairwise disjoint.

\begin{corollary} 
Assume that $n\ge3$ and
$\kappa:=\inf_i\kappa_i>\frac{n(6-n)-4}{4(n-1)}.$
 Then 
\begin{align}\lambda_1^{(p)}&\ge
\alpha_t\min\bigg\{K, \inf_{i\in{\mathfrak I}}\bigg\{ (n-1)\ell_i-
\frac13(1-\kappa)^+\Big[(n+1)\frac{\pi^2}{2\rho_i^2}+(n-2)\ell_i\Big]\bigg\}
\bigg\}.
\end{align}
with $$
t={(1-\kappa)^+}\frac{4(n-1)}{n(n-2)}<1, \qquad
\alpha=
\frac n{n-1}\bigg[1+\frac{\frac t{1-t}}{(n-1)^2}
\bigg]^{-1}.$$
In particular, if $\ell_i=\frac{\pi^2}{4\rho_i^2}$ for all $i$ and $\kappa\le 1$. Then 
\begin{align}\nonumber\lambda_1^{(p)}&\ge
\alpha\min\Big\{K, \ \big(n\kappa-1\big)\inf_i\ell_i\Big\}.
\end{align}

\end{corollary}

\begin{proof}
Put $L_i:=\frac{\pi^2}{4\rho_i^2}\ge \ell_i$ 
and
\begin{align*}k:=&\begin{cases}
K,\quad &\text{in } M_0\\
(n-1){\ell_i}-(n-2)\frac{(1-\kappa)^+}{\sin_{\ell_i}^2d(z,.)},\quad \ &\text{in } M_i,
\end{cases}\\
\vartheta:=&\begin{cases}
0,\quad &\text{in } M_0\\
\Big(\frac{n-2}2\Big)^2\frac1{\underline\tan^2_{L_i}d(z,.)}-\frac{n-2}2{L_i}\qquad &\text{in } M_i
\end{cases}
\end{align*}
Then $\Ric^M\ge k$ on $M$ and $-\Delta\ge \vartheta$. Thus as before
\begin{align*}\lambda_1^{(p)}&\ge\alpha_t \lambda_0\Big(-\frac{tn}{n-1}\Delta+k\Big)
\\
&\ge\alpha_t \inf_M\Big(\frac{tn}{n-1}\vartheta+k\Big)\\
&=\ldots\\
&=
\alpha_t\min\bigg\{K, \inf_{i\in{\mathfrak I}}\bigg\{ (n-1)\ell_i-
\frac13(1-\kappa)^+\Big[(n+1)\frac{\pi^2}{2\rho_i^2}+(n-2)\ell_i\Big]\bigg\}
\bigg\}.
\end{align*}
\end{proof}

\subsubsection*{Appendix}

\begin{lemma}\label{Euler} For $0\le \ell\le L$ and $0<r\le \pi/(2\sqrt L)$,
$$0\le \frac1{\tan^2_\ell(r)}-\frac1{\tan^2_L(r)}\le \frac23(L-\ell).$$
\end{lemma}
\begin{proof} Consider the function 
$$f(t):=\frac t{\sin^2(\sqrt t)}$$
on $(0,\pi^2/4)$. It satisfies
\begin{align*}
f'(t)&= 
\frac{1-\frac{\sqrt t}{\tan(\sqrt t)}}{\sin^2(\sqrt t)}
\ge \frac{\frac13t}{\sin^2(\sqrt t)}\ge\frac13.
\end{align*}
Thus for all $0<s<t<\pi^2/4$,
$$f(t)-f(s)\ge \frac13 (t-s).$$
Choosing $s=\ell r^2$ and $t=L r^2$ yields 
\begin{align*}
 \frac1{\tan^2_\ell(r)}-\frac1{\tan^2_L(r)}&= \frac1{\sin^2_\ell(r)}-\frac1{\sin^2_L(r)}+(L-\ell)\\
 &=\frac1{r^2}\left[\frac s{\sin^2(\sqrt s)}-\frac t{\sin^2(\sqrt t)}\right]+(L-\ell)\le \frac23(L-\ell).
\end{align*}
\end{proof}

\section{Further Curvature  and Spectral Properties}
\subsection{Discrete Spectrum}
Assume that $(M,g)$ is a closed Riemannian manifold and that it has a conical singularity at $z\in M$ such that
  \begin{itemize}
\item
$\Ric^M\ge K$ on $M_0:=M\setminus  B_{\rho}(z)$
\item $M_1:=B_{\rho}(z)\simeq [0,\rho)\times_{f_\ell} N$  and $\Ric^{N}\ge (n-2)\kappa$.
\end{itemize}
with numbers $K,\ell,\rho\in\R$ where $\rho< 2R_\ell$,
  and with a closed $(n-1)$-dimensional manifold $N$.

\begin{theorem} Assume that 
$n\ge 3$
and 
$\kappa>0$. Then the spectrum of $\Delta$ is discrete and 
$$\text{\rm spec}_{ess}\left(-\Delta\right)=\emptyset.$$
\end{theorem}

\begin{proof}
Let us first address the non-trivial case where $\kappa<1$.
We will represent the Laplacian $\Delta$ on $(M,g)$ as a `small' perturbation' of the Laplacian $\Delta^*$ on a modified Riemannian manifold $(M,g^*)$ for which we can verify a $\BE_2(K^*,n)$ condition with some $K^*\in\R$.

The space $(M,g^*)$ will be constructed as follows: we replace the metric $g$ on the set $M_1\simeq[0,\rho)\times_{f_{\ell}} N$ by the metric $g^*$ of the warped product
$M^*_1:=[0,\rho)\times_{f_*} N$ where $f_*:[0,\rho)\to\R$ is any strictly increasing $\C^2$-function with
$\sqrt\kappa\, f_\ell\le f_*\le f_\ell$ and 
$$f_*(r)=\begin{cases}
\sqrt\kappa\, f_\ell(r), \qquad&\text{on }[0,\rho/4)\\
f_\ell(r), \qquad&\text{on }(\rho/2,\rho).
\end{cases}$$
Then it is easy to verify that $\Ric^{M^*}\ge K^*$ on $M^*\setminus\{z\}$ for some $K^*\in\R$ and, even more, $(M^*,g^*)$ satisfies the $\BE_2(K^*,n)$-inequality. Since in addition $M^*$ is compact, for every $\alpha>0$ the resolvent
operator
$G_\alpha^*:=(\alpha-\Delta^*)^{-1}$ is a compact self-adjoint operator on $L^2(M^*,m^*)$. 
Note that $dm^*= \eta^2 dm$ with the bounded function
$$\eta:= \left(\frac{f_*}{f_\ell}\right)^{\frac{n-1}2}.$$
 Consider the unitary transformation
$$T: L^2(M^*,m^*)\to L^2(M,m), \ u\mapsto u\,\eta.$$
Then 
$$G_\alpha^\flat:=T\circ G_\alpha^*\circ T^{-1}$$ is a compact self-adjoint operator on $L^2(M,m)$. 
Put $\Delta^\flat:=T\circ \Delta^*\circ T^{-1}$. Then
$G_\alpha^\flat=(\alpha-\Delta^\flat)^{-1}$ on $L^2(M,m)$. 
Moreover, 
$$\Delta u=\partial_r^2 u+(n-1)\frac{f_\ell'}{f_\ell}\partial_r u+ \frac1{f_\ell^2}\Delta^N u$$
whereas
$$\Delta^\flat u=\partial_r^2 u+(n-1)\frac{f_\ell'}{f_\ell}\partial_r u+ F\cdot u+\frac1{f_*^2}\Delta^Nu$$
with
$$F(r):=\frac{n-1}2\left[\frac{f''_\ell}{f_\ell}-\frac{f''_*}{f_*}-(n-1)\frac{f_\ell' f_*'}{f_\ell f_*}+\frac{n-3}2\frac{{f'_\ell}^2}{f_\ell^2}
+\frac{n+1}2\frac{{f'_*}^2}{f_*^2}
\right].$$
Note that $F\equiv 0$ on $[0,\rho/4)\cup (\rho/2-\rho)$ and thus $F\ge -C$ for some $C\in\R_+$.
Therefore,
\begin{align*}
0\le \alpha-\Delta^\flat&\le-\left(\partial_r^2 +(n-1)\frac{f_\ell'}{f_\ell}\partial_r +\frac1{f_*^2}\Delta^N\right)+(C+\alpha)\\
&\le -\frac1\kappa\left(\partial_r^2 +(n-1)\frac{f_\ell'}{f_\ell}\partial_r +\frac1{f_\ell^2}\Delta^N\right)+(C+\alpha)\\
&\le \frac1\kappa(C+\alpha-\Delta)
\end{align*}
for every $\alpha>0$, and hence,
$$0\le(C+\alpha-\Delta)^{-1}\le \frac1\kappa ( \alpha-\Delta^\flat)^{-1}.$$
Compactness of $( \alpha-\Delta^\flat)^{-1}$ thus implies compactness of  $(C+\alpha-\Delta)^{-1}$ and this in turn implies 
that the spectrum of $\Delta$ is discrete. In particular,  the essential spectrum is empty.

Finally, let us consider the case $\kappa\ge1$. Then without changing the metric we easily can verify the $\BE_2(K^*,n)$ condition with some $K^*$ for $(M,g)$.
Together with the compactness of  $M$, this implies compactness of the resolvent
operator $(\alpha-\Delta)^{-1}$ and thus discreteness of the spectrum of $\Delta$.
\end{proof}


\subsection{The Taming Semigroup}

In various approaches to functional inequalities and spectral estimates on  singular spaces $M$ with synthetic lower bounds $k:M\to\R$, a crucial role is played by the taming operator $-\Delta +k$ and the associated \emph{taming semigroup}
$$P^k_tu:=e^{-(-\Delta+k)t},$$ 
a Schr\"odinger semigroup with the Ricci bound entering as potential. According to the celebrated Feynman-Kac formula it is given in probabilistic terms as
$$P^k_tu(x)={\mathbb E}_x\Big[e^{-\int_0^t k(B_{2s})ds} u(B_{2t})
\Big]$$
where $(B_t)_{t>0}$ denotes Brownian motion on $M$ generated by $\frac12\Delta$. The crucial assumption in \cite{ERST} is that
$(P^k_t)_{t>0}$ defines an exponentially bounded semigroup on $L^\infty(M)$.

For manifolds with negatively curved conical singularities, however, the taming semigroup will never be bounded on $L^\infty(M)$.
The best one can hope for is boundedness on $L^2(M,m)$. Let us illustrate this in the most simple case.

\begin{example} Let $(M,g)$ be a closed 3-dimensional Riemannian manifold such that $g$ is smooth on $M\setminus\{z\}$ and
$M\supset B_\rho(z)\simeq [0,\rho)\times_r \mathbb S^2_R$ for some $R\in (0,\infty)$.
\begin{itemize}
\item[(i)] \ $P^k_t: L^\infty\to L^\infty$ bounded \quad $\Longleftrightarrow\quad  R\le1$.
\item[(ii)] \ $P^k_t: L^2\to L^2$ bounded \quad $\Longleftrightarrow\quad  R\le R^\natural:=\sqrt{\frac86}$.
\item[(iii)] \ $P^k_tu\equiv +\infty$ for any $u\ge0, u\not\equiv 0$ \quad $\Longleftrightarrow\quad  R> R^\natural$.
\item[(iv)] \ $k$ is $n$-admissible \quad $\Longleftrightarrow\quad  R< R^\sharp:=\sqrt{\frac85}$. 
\item[(v)] \ $\BE'_1(k,n)$ holds \quad $\Longleftrightarrow\quad  R< R^\sharp$.
\end{itemize}
\end{example}
 Occasionally also the semigroup
$
e^{-(-\frac12\Delta+k)t}$
generated by $-\frac12\Delta +k$ is employed.
Again, this semigroup is unbounded on $L^\infty$ whenever $R>1$. It is bounded on $L^2$ if and only if $R\le R^\flat:=\sqrt{\frac87}$.
Note that $R^\sharp>R^\natural>R^\flat$.

\subsection{Weighted Spaces}
Lower Ricci bounds of the form
$$k(x)=C_0-\frac{C_1}{d^2(x,z)}$$
not only appear at conical singularities but most naturally also at singularities of weights.
Given a smooth 
Riemannian manifold $(M,g)$ and a (sufficiently regular) weight $\rho:M\to [0,\infty]$, the Bakry-\'Emery Ricci tensor for the metric measure space $(M,d_g, \rho\vol_g)$ is given by
$$\Ric^{M,\rho}_x(\xi,\xi):=\Ric_x^M(\xi,\xi)-\text{Hess}_x(\log\rho)(\xi,\xi).$$
\begin{example} Consider $M=\R^n$ with $g=g^{Euclid}$ and $\rho(x)=|x|^\alpha$ for some $\alpha\in\R$. Then 
$$\inf_\xi\frac{\Ric_{x}^{M,\rho}(\xi,\xi)}{g_{x}^M(\xi,\xi)}=
\inf_{\xi}\frac{\alpha}{|x|^4\,|\xi|^2}\Big[-|x|^2\,|\xi|^2+2\langle x,\xi\rangle^2
\Big]=
-\frac{|\alpha|}{|x|^2}.$$
 
\end{example}
\subsection{Grushin-type Spaces}

Further important examples with singular Ricci bounds are provided by Grushin-type spaces.
\subsubsection{Riemannian Grushin (Half-) Spaces}

Consider $$M=\R\times_f \R^{n-1}, 
\qquad f(y)=|y|^{-\alpha}$$
with $\alpha>0$
such that
$$\mathcal E^M(u)=\int\left[|\nabla_y u|^2+\frac1{f^2(y)}|\nabla_zu|^2
\right]d{\frak L}^{1}(y)\, f^{n-1}(y) d{\frak L}^{n-1}(z)
$$
and
$\Delta^M u=\partial^2_y u+ (n-1)\frac{f'}{f}\partial_y u+\frac1{f^2}\Delta_zu$.

%
\begin{proposition} Then \begin{itemize}
\item[(i)]
$M^*:=(\R\setminus\{0\})\times_f \R^{n-1}$ is a (disconnected) smooth Riemannian manifold with 
$$k(y,z):=\inf_\xi\frac{\Ric_{y,z}^M(\xi,\xi)}{g_{y,z}^M(\xi,\xi)}=-\frac{\alpha}{|y|^2}.$$

\item[(ii)]
The function $k$   is n-admissible if and only if $\alpha<\frac{n}{4(n-1)}$. In particular, it is 2-admissible if and only if $\alpha<\frac12$.
\end{itemize}
The same assertions hold true with $M$ replaced by $M_+:=\R_+\times_f \R^{n-1}$.
\end{proposition}

\begin{proof} (i) Lemma \ref{warp-form}.
(ii) Hardy inequality on $\R$ (resp. $\R_+$).
\end{proof}

\subsubsection{Riemannian Grushin-type Spaces}
Consider $$M=\R^{j}\times_f \R^{n-j},  \qquad f(y)=|y|^{-\alpha}$$
with $j\ge 2$, $\alpha>0$ such that
$$\mathcal E^M(u)=\int\left[|\nabla_y u|^2+\frac1{f^2(y)}|\nabla_zu|^2
\right]d{\frak L}^{j}(y)\, f^{n-j}(y) d{\frak L}^{n-j}(z)
$$
and
$\Delta^M u=\Delta_y u+ \frac{n-j}{f}\nabla_y f\nabla_y u+\frac1{f^2}\Delta_zu$.

\begin{proposition} Then \begin{itemize}
\item[(i)]
$M^*:=(\R^{j}\setminus\{0\})\times_f \R^{n-j}$ is a  smooth Riemannian manifold with 
$$k(y,z):=\inf_\xi\frac{\Ric_{y,z}^M(\xi,\xi)}{g_{y,z}^M(\xi,\xi)}=-(n-j)\frac{\alpha^2}{|y|^2}.$$

\item[(ii)]
The function $k$   is n-admissible if and only if $\alpha<\frac{(j-2)^2n}{4(n-1)}$.
\end{itemize}
\end{proposition}

\begin{proof} (i) \cite{Oneill}, \cite{Ketterer-cones}.
(ii) Hardy inequality on $\R^j$.
\end{proof}

\subsubsection{Grushin-type Spaces with Lebesgue Measure}
Consider $$M=\R^{j}\times_f \R^{n-j},  \qquad f(y)=|y|^{-\alpha}, \qquad dm=d{\frak L}^n$$
with $j\ge 2$, $\alpha>0$ such that
$$\mathcal E^M(u)=\int\left[|\nabla_y u|^2+\frac1{f^2(y)}|\nabla_zu|^2
\right]d{\frak L}^{j}(y)\,  d{\frak L}^{n-j}(z)
$$
and
$\Delta^M u=\Delta_y u+\frac1{f^2}\Delta_zu$.

\begin{proposition} Then 
$M^*:=(\R^{j}\setminus\{0\})\times_f \R^{n-j}$ with $dm=d{\frak L}^n$ is a  weighted smooth Riemannian manifold with 
$$\inf_\xi\frac{\Ric_{y,z}^{M,\rho}(\xi,\xi)}{g_{y,z}^M(\xi,\xi)}\ge-(n-j)\frac{\alpha(\alpha+1)}{|y|^2}.$$

\end{proposition}

\begin{proof} The given space is a weighted manifold with weight $\rho(y)=f^{j-n}(y)=|y|^{\alpha(n-j)}$. Thus
\begin{align*}
\inf_\xi\frac{\Ric_{y,z}^{M,\rho}(\xi,\xi)}{g_{y,z}^M(\xi,\xi)}&\ge
\inf_\xi\frac{\Ric_{y,z}^{M}(\xi,\xi)}{g_{y,z}^M(\xi,\xi)}+\inf_\xi\frac{\text{Hess}_{y,z}(-\log\rho)(\xi,\xi)}{g_{y,z}^M(\xi,\xi)}\\
&\ge -(n-j)\frac{\alpha^2}{|y|^2}+\inf_\xi\frac{\text{Hess}_{y}(-\log\rho)(\xi,\xi)}{g_{y}(\xi,\xi)}\\
&= -(n-j)\frac{\alpha^2}{|y|^2} -(n-j)\frac{\alpha}{|y|^2}.
\end{align*}
\end{proof} 

\paragraph{Acknowledgements.}
The author gratefully acknowledges financial support  by the Deutsche Forschungsgemeinschaft through the project `Ricci flows for non-smooth spaces, monotonic quantities, and rigidity' within the SPP 2026 `Geometry at Infinity'  and through the excellence cluster
`Hausdorff Center for Mathematics'.

\end{document}